\documentclass[12pt, titlepage, a4paper,  draft]{article}

\usepackage{amssymb, latexsym,amsfonts,amsmath,amsthm}
\usepackage{verbatim}

\usepackage{url}

\usepackage[usenames,dvipsnames]{xcolor}

{\makeatletter \@addtoreset{equation}{section}}

\newcommand{\diag}{\mathop{\mathrm{diag}}\nolimits}
\newcommand{\Ker}{\mathop{\mathrm{Ker}}\nolimits}
\newcommand{\IIm}{\mathop{\mathrm{Im}}\nolimits}
\newcommand{\rank}{\mathop{\mathrm{rank}}\nolimits}

\newcommand{\Inv}[1][\,]{\mathop{\mathrm{Inv}#1}\nolimits}

\newcommand{\Z}{{\mathbb{Z}}}

\newcommand{\h}{ \mathop{ \mathrm{h} {} }\nolimits }
\newcommand{\e}{ \mathop{ \mathrm{e} {} }\nolimits }

\newcommand{\beq}{\begin{equation}}
\newcommand{\eeq}{\end{equation}}

\newcommand{\LR}{\Leftrightarrow}
\newcommand{\Ra}{\Rightarrow}

\newcommand{\bpm}{\begin{pmatrix}}
\newcommand{\epm}{\end{pmatrix}}

\newcommand{\iiloplus}
{\hbox{  {\boldmath{ \large{$\oplus$}
                   }
         }
       }
}

\newtheorem{theorem}{Theorem}[section]
\newtheorem{lemma}[theorem]{Lemma}
\newtheorem{definition}[theorem]{Definition}
\newtheorem{example}[theorem]{Example}
\newtheorem{corollary}[theorem]{Corollary}
\newtheorem{prop}[theorem]{Proposition}

\newtheorem{obs}[theorem]{Observation}

\newtheorem{remark}[theorem]{Remark}

\newtheorem{conj}[theorem]{Conjecture}

\definecolor{bbl}{rgb}{0.2 , 0.1  , 0.8 }

\definecolor{ungu}{rgb}{0.8 , 0  , 0.7 }
 \definecolor{exp}{rgb}{0.85,0.7,0}
  \definecolor{msk}{rgb}{0.65,0.3,0.3}

\definecolor{msk2}{rgb}{0.75,0.3,0.1}
\definecolor{r}{rgb}{0.75,0.1,0.4}

\definecolor{exp2}{rgb}{0.85,0.6,0.1}

\definecolor{exp4}{rgb}{0.85,0.3,0.5}

\definecolor{vvs}{rgb}{0.5,0.4,0.7}

\definecolor{kf}{rgb}{0.9,0.0,0.1}
\definecolor{gr}{rgb}{0.3,0.0,0.7}

\begin{document}

\title
{Characteristic subspaces and  hyperinvariant frames}

\author{Pudji Astuti
\\
Faculty of Mathematics\\
 and Natural Sciences\\
Institut Teknologi Bandung\\
Bandung 40132\\
Indonesia
\thanks{The work of the first author was supported by the program  ``Riset 
dan Inovasi KK ITB''  of  the  Institut Teknologi Bandung.}  
         \and
Harald~K. Wimmer\\
Mathematisches Institut\\
Universit\"at W\"urzburg\\
97074 W\"urzburg\\
Germany
}

\date{}

\maketitle

\begin{abstract}

Let $f$ be an endomorphism of a finite dimensional vector 
space  $V$  over a field $K$.  
An $f$-invariant subspace 
is called
 hyperinvariant (respectively characteristic)   if 
it is invariant under all endomorphisms
(respectively  auto\-morphisms) that commute with $f$.
We assume $|K| = 2$, 
since all characteristic subspaces are  hyperinvariant
if $|K| > 2$. 
The hyperinvariant hull $W^h$ of  a subspace  $ W$ of $ V$
is defined to be the smallest hyperinvariant  subspace of  $V$
that contains $ W$, the hyperinvariant kernel $W_H$  of $ W$  
is the largest  hyperinvariant  subspace of $V$ that is 
contained in $W$, and the pair $( W_H, W^h) $ is 
the hyperinvariant frame of $W$.  In this paper we study
 hyperinvariant frames of characteristic non-hyperinvariant 
subspaces $W$. 
We show that all invariant subspaces in the interval
$[  W_H, W^h ]$ are characteristic. 
We use this result  for the construction of  characteristic non-hyperinvariant 
subspaces.

\vspace{10mm}
\noindent
{\bf Mathematical Subject Classifications (2010):}
 	15A04, 
47A15, 
15A18,   
20K01.   
\vspace{.5cm}

\noindent
 {\bf Keywords:}  characteristic sub\-spaces,
 hyperinvariant subspaces,
  invariant subspaces,   
characteristic hull, hyperinvariant hull,  hyperinvariant frame, 
exponent, height, 
fully invariant subgroups, lattice intervals.

\vspace{.2cm}
\flushleft{
{\textsf e-mail:}~~\texttt{\small wimmer@mathematik.uni-wuerzburg.de} }\\

\flushleft{
{\textsf e-mail:}~~\texttt{\small
pudji@math.itb.ac.id
} }

\end{abstract}


\section{Introduction} \label{sct.ntrd}
Let 
$V$ be  an $n$-dimensional vector space over a field $K$ and let 
$f: V \to V$ be $K$-linear. 
A subspace $X \subseteq V$ is said to be {\em{hyperinvariant}}
under  $f$   
(see e.g.\ \cite[p.\ 305]{GLR}) 
if it  remains invariant under
all endomorphisms of $V$ that commute with $f$.
If $X$ is an   $f$-invariant subspace of $V$ and if
$X$ is  invariant
under
 all automorphisms of $V$ that commute with $f$,
then \cite{AW1}  
  we say that    $X$ is   {\em{characteristic}} (with respect to $f$).
Let $ {\rm{Inv}}( V,f) $, \,$   {\rm{Hinv}}(V,f)$, and
$  {\rm{Chinv}} (V,f)  $ be  sets of  invariant,   hyperinvariant and
  characteristic  subspaces of  $V$, respectively.
These sets are lattices (with respect to set inclusion), and
 \[   {\rm{Hinv}} (V,f)
 \subseteq   {\rm{Chinv}} (V,f)   \subseteq   {\rm{Inv}}( V,f)  .
\]
The structure of the  lattice  \,$   {\rm{Hinv}} (V,f) $\,
is well understood (\cite{Lo},  \cite{FHL},
\cite{Lo2},  \cite[p.\,306]{GLR}). 
If $ f $ is nilpotent then 
 \,$   {\rm{Hinv}} (V,f)  $\, is the 
sublattice of $ {\rm{Inv}}( V,f) $
generated by
\[
  \Ker f ^k,  \,\, \IIm f^k ,  \: k = 0, 1, \dots , n.
\]
It is   known  (\cite{Sh}, \cite[p.\,63/64]{Kap},   \cite{AW1})
that each  characteristic subspace
is hyper\-invariant if  $|K| > 2$.
Hence, only if $V$ is a vector space over the
field $ K = GF(2) $ there may  exist  $K$-endomorphisms  $f$  of  $V$
with   characteristic subspaces that are not   hyperinvariant.

If the   characteristic
  polynomial of $f $ splits over $K$ (such that all eigen\-values of
$f$ are in $ K$)
 then 
  one can restrict the
  study of 
 hyperinvariant and of characteristic  subspaces to the case where
$ f $ has only one eigenvalue, and therefore to the
case where $f$ is nilpotent.
Thus, throughout this paper we shall assume \,$ f^n = 0 $.
Let 
\,$  \Sigma(\lambda)  =  \diag (1, \dots , 1, 
 \lambda^{t_1}, \dots , \lambda^{t_m}) \in K^{n \times n}[\lambda] $\,
be the Smith normal  form of $f$ such that $ t_1 + \cdots +t_m = n$.  
We say that 
$ \lambda ^{t_j}$  is  an  {\em{unrepeated}} elementary divisor  of $f$
if  it  appears exactly once in $\Sigma(\lambda) $. 
We note the following result which is due to Shoda 
(see also \cite[Theorem~9, p.\,510]{Bae} and  \cite[p.\,63/64]{Kap}).

\begin{theorem}    \label{thm.vnpsa}  
{\rm{\cite[Satz~5, p.\,619]{Sh}}}      
Let $ V $ be a finite dimensional vector space over
  the field  \,$ K =  GF(2) $ and let $ f : V \to V $ be nilpotent.
The following statements are equivalent.
\begin{itemize}
\item[{\rm{(i)}}]
There exists a  characteristic subspace of  \,$V$
 that is not  hyper\-invariant.
\item[{\rm{(ii)}}]
The map $f$ has  unrepeated elementary divisors $\lambda^R $
and $\lambda^S $  such that \,$ R + 1  < S $. 
\end{itemize}
\end{theorem}

Suppose $X \in  {\rm{Inv}}( V,f) $. Let 
$X_H $  be the largest element in $ {\rm{Hinv}}( V,f)  $ 
 such that $ X_H \subseteq X $,  
and  let $X^h $  be the smallest   element in $ {\rm{Hinv}}( V,f)  $ 
  such that  $ X \subseteq X^h $.  Then  $  X_H  \subseteq X  \subseteq X^h$. 
We call $X_H $ and $X^h $   the {\em{hyperinvariant kernel}} 
and the {\em{hyperinvariant hull}}  of $X$, 
 respectively, and 
we say that  the pair $(X_H , X^h) $  is the {\em{hyperinvariant frame}} 
of $X$. 
Thus, $X \in  {\rm{Chinv}}( V,f)  $  is not hyperinvariant if and only if 
$  X_H   \subsetneqq X  \subsetneqq X^h$.  
In this paper we study  pairs $(X_H , X^h)  $ which can 
occur as hyperinvariant frames of characteristic non-hyperinvariant subspaces.
We shall see that all elements  of   corresponding intervals
$[X_H , X^h]$ are characteristic subspaces.  
We regard this fact as  essential for further investigations of the 
lattice structure of $  {\rm{Chinv}} (V,f)  $.
Our main results are contained in  Sections 
~\ref{sbs.frm} - \ref{sbs.jso}.
In Section~\ref{sbs.notat}  we introduce 
basic concepts such as exponent and height, generator tuples 
 and  the group of $f$-commuting automorphisms. Related auxiliary 
material is gathered together in Section~\ref{sbs.gen}.
Hyperinvariant subspaces are reviewed in  Section~\ref{sbs.hyp}.

We remark that  Shoda \cite{Sh}  deals with   abelian groups. 
But 
it is known (see e.g.\  \cite{BCh}) that
in many instances 
methods or concepts  
of abelian group theory can be applied to linear algebra
if they are translated 
to modules over principal ideal domains and then specialized
to $K[\lambda]$-modules. 
On the other hand  there are parts  of linear algebra  that  can 
  be interpreted in the framework 
 of  abelian group theory.  In our case  
the language would change, and 
proofs    would carry over  almost verbatim  to finite abelian $p$-groups \cite{FuII}.
Instead  of  hyperinvariant subspaces one would deal with
 subgroups that are   fully invariant,  
and instead of characteristic non-hyperinvariant subspaces  with 
irregular characteristic subgroups \cite{Bae}, \cite{KR}. 


\subsection{Notation and basic concepts}  \label{sbs.notat}

Let $ x \in V $.   Define $ f^0 x  = x $.
The smallest nonnegative integer $\ell$
 with
$f^{\ell} x = 0$
is  called
the {\em{exponent}} of $x$. We write  $\e(x) = \ell$.
A  nonzero vector  $x $
  is said to have
\emph{height} $q $ if $x \in f^q V$
and $x \notin f^{q+1} V$.
In this case we write $\h(x) = q$. We set $ \h ( 0 ) =   \infty $.
The group of  automorphisms of $V$
that commute with $ f $ will be denoted by
 ${\rm{Aut}}(V,f) $. Then  ${\rm{Aut}}(V,f)  \subseteq  {\rm{End}}(V,f) $, where
 $  {\rm{End}}(V,f) $ is the algebra of
all endomorphisms of $V$ that commute with $f$. 
Clearly, if  $ \alpha \in {\rm{Aut}}(V,f) $ then
  $ \alpha ( f^i  x ) = f^i(  \alpha x ) $  for all
$x \in  V$. Hence 
$  \e( \alpha x ) = \e(x)$ and $  \h( \alpha x ) = \h(x) $ 
for all   $x \in V, \,  \alpha \in {\rm{Aut}}(V,f)$.  
We set $  V[ f^j ] =  \Ker f^j $, $j \ge 0 $.
Thus,  the assumption  $ f^n = 0 $ implies $ V =  V[ f^n]$. 
Let
\begin{multline*}
 \langle x  \rangle \,  = \,  {\rm{span}} \{ f^i x , \, i \ge 0 \}
=
\\
\{c_0x + c_1 f x + \cdots + c_{n-1} f^{n-1}x ;
\,
  c_i \in K, \, i = 0, 1, \dots , n-1 \}
\end{multline*}
 be  the  $f$-cyclic subspace  generated by $ x $.  
To $ B  \subseteq V$ we associate the subspaces 
\,$  \langle B \rangle  = 
 \sum _{b \, \in \,  B } \,  \langle \,  b  \, \rangle  $,  and  
\[  B  ^c  =  \cap \{  W \in {\rm{Chinv}} (V,f) ; \,B \subseteq W \} , \: \; 
 B  ^h = \cap \{  W \in {\rm{Hinv}} (V,f) ; \,B \subseteq W \}. 
\]
Then  
\[ 
   \big\langle   B  \big\rangle ^c =  \big\langle  \alpha b ; \, b \in B,
\alpha \in   {\rm{Aut}}(V,f)  \big\rangle =
\sum \nolimits  _{b \in B} \langle   b \rangle ^c , 
\]
and  
\[
   \big\langle   B  \big\rangle ^h =  \big\langle  \eta b ; \, b \in B,
\eta \in   {\rm{End}}(V,f)  \big\rangle =
\sum \nolimits  _{b \in B} \langle   b \rangle ^h,
\]
and  \,$    \langle B \rangle ^c \subseteq   \langle B \rangle ^h $.  
We call  the subspaces $ B  ^c $  and $ B  ^h $ the 
{\em{characteristic hull}}  and the {\em{hyperinvariant hull}} of $  B  $,
respectively. 
A  subspace $X$ is  hyperinvariant if and only if 
$ X = \langle   X \rangle ^c =  \langle   X  \rangle ^h $.

Suppose  $ \dim \Ker f = m $.  Let $ \lambda ^{t_1} , \dots , \lambda ^{t_m} $
be the elementary divisors of $f$ such that
$ t_1 + \cdots + t_m = \dim V  $. 
Then $V$ can be decomposed into a direct sum of $f$-cyclic  subspaces
$ \langle u_j \rangle $ such that
\beq \label{eq.vdeco} 
 V =
  \langle u_1 \rangle
 \,  \oplus \,
\cdots \,  \oplus    \langle u_{ m } \rangle \quad \text{and} 
\quad \e(u_j) = t_j ,  \: j= 1, \dots , m . 
\eeq  
Let $\pi _j : V \to V$, $ j = 1, \dots , m$,   be  projections 
be defined by 
\[\IIm  \pi _j =  \langle u_j\rangle \quad {\rm{and}} \quad  \Ker \pi _j 
= 
\langle u_1, \dots , u_{j-1} ,  u_{j+1} , \dots , u_m \rangle .
\]
Note that $ \pi_j \in  {\rm{End}}(V,f) $.
If \eqref{eq.vdeco} holds 
and 
\beq \label{eq.teug}  
0 < t_1 \,\,   \le \,\,  \cdots \,\,  \le \,\,  t_m , 
\eeq 
then we  say that 
  \,$ U = ( u_1, \dots , u_m) $\,    is  a
{\em{generator tuple}} of $V$
(with respect \mbox{to $f$}). 
The  tuple $(t_m , \dots , t_1) $ of exponents  - written
in nonincreasing order -  is known as {\em{Segre characteristic}} 
of $f$. 
The set of  generator tuples  of $V$ will be denoted by
 $ \mathcal{U} $.
We call $ u \in V $
 a {\em{generator}}  of $V$
(see also \cite[p.4]{FuI})  if 
$u \in U $ for some $U \in  \mathcal{U} $. 
Then 
$u \in V $ is a generator if and only if $u\ne0$ and
\[
  V =   \langle u \rangle \oplus V_2  \quad \text{for some}
  \quad V_2 \in  {\rm{Inv}}( V,f) . 
\]
Unrepeated elementary divisors  $\lambda ^{t_i} $ 
and  corresponding generators  
will play a crucial role in this paper. Therefore we single out the 
corresponding unrepeated 
exponents  $t_i$ and define a set of indices 
\[
I_u =  \{ i ; \,     t_i \ne t_k    \:  \:  \hbox{if} \:   \:   k \ne i,  \, 1 \le k \le m \} .
\]
Hence  we have  
$ i \in I_u $ if and only if 
\beq \label{eq.ulm}
 \dim \big(  V[f] \cap  f ^{t_{i} -1 }V \,\,  /  \, \, V[f]  \cap  f ^{t_i}  V \big) = 1 .
\eeq 
The left-hand side of \eqref{eq.ulm} 
is the $(t_i -1)$-th Ulm invariant of $f$ (see \cite[p.\,154]{FuI},  \cite[p.27]{Kap}).
We say that  a generator   $u$  is unrepeated if  \,$ \e(u) = t_i $ and 
$ i \in I_u$. 


\subsection{Hyperinvariant subspaces} \label{sbs.hyp}

Let  $ U = (u_1, \dots , u_m) \in \mathcal{U} $ be a generator tuple 
such that \eqref{eq.vdeco} and  \eqref{eq.teug} hold. 
Define  \mbox{$ \vec{t}  = (t_1 , \dots , t_m) $}. 
Let  $\mathcal{L}( \vec t \, ) $  be the set of $m$-tuples 
$\vec r = (r_1 , \dots , r_m ) $ 
of integers 
satisfying
\beq \label{eq.zwrug}  
0  \;  \le \;   r_1 \;  \le  \; \cdots \;   \le \;   r_m  \:\:\,  {\rm{and}} \:\:\, 
0 \;   \; \le \;  t_1 - r_1  \;   \le \;  \cdots  \; \le \;  t_m -  r_m  .
\eeq 
We write \, $ \vec r   \preceq \vec s $ \, 
if  \,$ \vec r= ( r_j )_{j=1} ^m $, $ \vec  s = (s_j ) _{j=1} ^m  \,  \in 
\, \mathcal{L}(\vec t \, \, ) $\,  
 and 
 \,$r_j \le s_j$, \, $ 1 \le j \le m$.  
Then  $\big( \mathcal{L}(\vec t \, ),   \preceq \! \big) $ is a
 lattice.  The following theorem is due to Fillmore, Herrero and 
Longstaff \cite{FHL}.  We  refer to \cite{GLR} for a proof.
A related result concerning  fully invariant subgroups of abelian $p$-groups
is Theorem~2.8 in \cite{BEis}.

\begin{theorem}   \label{thm.fhl} 
Let $f : V \to V $ be nilpotent. 
\begin{itemize} 
\item[\rm{(i)}]
If $ \vec r \in  \mathcal{L}(\vec t \,) $,   
then 
\[ 
W(\vec r \, ) = f ^{r_1}V \cap V[ f^{t_1 - r_1} ]
 \,  + \cdots + \, 
 f^{r_m}V \cap V[ f^{t_m- r_m} ] 
\]
 is a hyperinvariant subspace.
Conversely, each
$ W \in  {\rm{Hinv}}(V,f)$ is of the form
$W = W(\vec r \, ) $ for some $ \vec r \in  \mathcal{L}(\vec t \, ) $.
\item[\rm{(ii)}] 
 If  $ \vec r \in  \mathcal{L}(\vec t \, ) $
then 
\,$ W(\vec r \, ) =
  f ^{r_1}  \langle u_1 \rangle  \,  \,  \oplus  \cdots  \,  \,
\oplus   
f ^{r_m}  \langle  u_m \rangle  $. 
\item[\rm{(iii)}]   The mapping  \,$ \vec r \mapsto W( \vec r \, ) $\, 
is a lattice isomorphism from \,$\big( \mathcal{L}(\vec t \, \, ),  \preceq \! \!\big) $\,
  onto $\big( {\rm{Hinv}} (V,f), \supseteq \!\!  \big)$.
\end{itemize} 
\end{theorem}

\medskip 

Let  $X \in   {\rm{Chinv}} (V,f)  $.
The  first part of Theorem~\ref{thm.wogn}  below  
deals with 
 the  hyperinvariant kernel  $X_H $ of $ X $.  In  \cite{Ba} 
the theorem   is used   to obtain  a description of the set
$   {\rm{Chinv}} (V,f)  \setminus  {\rm{Hinv}} (V,f) $.

\begin{theorem}  \label{thm.wogn} 
{\rm{\cite{AW2}}}  
Suppose  $X $ is a characteristic subspace of $V$.
Let  $ U = (u_1, \dots , u_m) \in \mathcal{U}$.
\begin{itemize} 
\item[\rm{(i)}] 
Then \,$
X_H  =  \oplus _{j=1} ^m \big( X \cap   \langle u_j \rangle  \big) $. 
 \item[\rm{(ii)}]
The subspace $X$ is hyperinvariant if and only if 
\beq  \label{eq.adin}
\pi_j X = X \cap \langle u_j \rangle \:\;  \mbox{for all} \:\;  j  \in \{1,  \dots, m \} . 
\eeq
\item[\rm{(iii)}]
If $ j \notin I_u$    
then 
\,$\pi_j X = X \cap \langle u_j \rangle $. 
If $ | I_u | \le 1 $, that is,  if $f$ has 
  at most one unrepeated elementary divisor,  then $X$ is hyperinvariant. 
\end{itemize}

\end{theorem} 

\medskip 
The following observation is  related to  
  \eqref{eq.adin}.

\begin{lemma}  \label{la.vble}  
 Let $ U =  \big( u_j  \big) _{j=1} ^m \in \mathcal U $.
Suppose   $X \in \Inv (V, f) $. 
Then  
\beq \label{eq.dgne} 
X \cap  \langle u_j \rangle  =  \langle f^{r_j}  u_j \rangle 
\:\: \mbox{and} \:\:  \pi _j X =  \langle f^{\mu_j}  u_j \rangle 
 \:\:  \: \mbox{with}  \,  \:\:    0 \; \le\; \ \mu_j \; \le \;  r_j   \; \le \;  t_j .
\eeq 
\end{lemma} 

\begin{proof} 
We have $ X \cap  \langle u_j \rangle \in \Inv (V, f) $,
and because of $ \pi _j f = f \pi _j $ we  also have $ \pi _ j X \in \Inv (V, f) $.
The invariant subspaces of  $ \langle u_j \rangle $ 
are   $ \langle f^{s}  u_j \rangle $, $ s = 0 , \dots , t_j$.
Hence      $X \cap  \langle u_j \rangle \subseteq  \pi _j X$ yields 
\eqref{eq.dgne}. 
 \end{proof} 

Let  $X \in   {\rm{Chinv}} (V,f)  $. 
Theorem~\ref{thm.hyhl} will show that    the numbers $r_j$ and $\mu_j$
in \eqref{eq.dgne} 
satisfy 
 $  r_j  -   \mu_j  \le   1$.  
We shall see that generators  $ u_j  \in U $
with 
\,$ \langle u_j \rangle \cap X  \;  \subsetneqq   \;  \pi _j X $\, 
require special attention.  
For that reason we associate to $X$ the set 
\[
J(X)  = \{ j ; \:    \langle u_j \rangle \cap X  \;   \subsetneqq   \;  \pi _j X \}. 
\]
We see  from 
Theorem~\ref{thm.wogn}(ii) that \,$X$\, is hyperinvariant if and only 
if  $ J(X) $ is empty. 
Moreover,  Theorem~\ref{thm.wogn}(iii) 
 implies  $J(X)  \subseteq I_u$.

\subsection{Generators and images under 
automorphisms}  \label{sbs.gen} 

In this section we derive an auxiliary result which we shall use
to determine the characteristic hull of subsets $B $ of $V$. 
Let $ U = (u_1, \ldots , u_m )   \,  \in \mathcal  U $ and $\alpha \in   {\rm{Aut}}(V,f) $.
Then $ \alpha U \in \mathcal  U $. 
On the other hand, if $ U^\prime =  ( u_1^\prime, \ldots , u_m^\prime  )\in \mathcal{U} $ then
a  mapping  $ \alpha : U \to  U^\prime $, 
$ \alpha :  u_j \mapsto u_j^\prime $,  $ j = 1, \dots, m $, 
can be extended  to a unique   $\alpha \in   {\rm{Aut}}(V,f) $. 
We first  note an  equivalent characterization of generators. 

\begin{lemma} {\rm{\cite[Lemma 2.6]{AW5} }}   \label{la.jtgn} 
Suppose $\lambda  ^t $ is an elementary divisor of $f$. 
Then $ x$ is a generator of \,$V$ with $ \e(x) = t $
if and only if  $f ^t x = 0 $ and 
\beq \label{eq.zmir} 
 h(x ) = 0 \quad and \quad 
 \h(f^{t-1}  x ) =  t-1 . 
\eeq
The condition \eqref{eq.zmir} 
is equivalent to   $\h (f^r x ) = r$, $ r= 0, 1, \dots , t -1$.
\end{lemma}

\medspace

Let  $ u_i \in U = (u_1, \ldots , u_m )  $  be an unrepeated generator. 
If  $ x \in V $ is a generator with $ \e(  x ) = \e (u_i ) $
then 
$ U ^\prime    =  (u_1, \ldots ,     u_{i -1} ,  x,    u_{i + 1} ,     \ldots ,      u_m )  
\in  \mathcal{U} $.
We say that the corresponding $f$-automorphism   $ \alpha :  U  \to U ^\prime $
exchanges  $ u_i $ by  $ x $, and we denote it by 
$\alpha ( u_i, x)$. 
The next lemma describes  the elements  $x$  that we  can choose for
 the replacement of an unrepeated  generator. 
Let $[x] = \{ \alpha x ; \, \alpha \in   {\rm{Aut}}(V,f)    \} $
 denote   the orbit of $x \in V $ under  ${\rm{Aut}}(V,f) $.

\begin{lemma} \label{la.nurpd}
Let $ U = (u_1, \ldots , u_m ) \in \mathcal{U} $ 
and suppose $ u_i \in U $ is unrepeated and $ \e(u_i) = t_i = t  $.
Then 
 $ x $ is a generator of  \,$V$ with  $ \e(x) = t $
if and only if  
\beq \label{eq.ngpl}
x=  u _i+ v + y  \quad with \quad  v \in  \langle   f  u  _i   \rangle 
\quad and \quad  y   \;   \in   \;  
\sum   \nolimits  _{ j =  1 , \,\, j \ne i} ^m \langle  u_j    \rangle [f ^t] . 
 \eeq 
Moreover 
\beq \label{eq.orb}
 [u_i]   =  u _i +    \langle   f  u  _i   \rangle  + \;
 \sum \nolimits _{ j = 1 , \, \, j \ne i  } ^m \langle  u_j    \rangle [f ^t] . 
\eeq 
\end{lemma}

\begin{proof} 
If \eqref{eq.ngpl} holds then $x$ satisfies  $ \e(x) = t $ and \eqref{eq.zmir}. 
Hence $x$ is a generator.  
Set  $ L^{\langle i \rangle}=  \langle u_1, \dots , u_{i-1}  \rangle $ and 
$G^{\langle i \rangle}   =  \langle u_{i+1} , \dots , u_m  \rangle $.  
Then  
\beq  \label{eq.rdem} 
 V =  L^{\langle i \rangle}  \oplus  
 \langle    u _i   \rangle  \oplus G^{\langle i \rangle} , 
\eeq 
and   
$V [ f^t]  =   
  L^{\langle i \rangle}  \oplus   \langle    u _i   \rangle  \oplus G^{\langle i \rangle} [ f^t] $. 
Let \,$ x = x_L + x_i + x_G   \in V $\, 
 be decomposed in accordance with \eqref{eq.rdem}.  
If $f^t x = 0 $ then  $  x_G  \in   G^{\langle i \rangle} [ f^t] $.  Moreover, if 
$x _ G \ne 0 $,  then 
$ \langle  u_j \rangle  [ f^t]  = \langle  f^{t_j - t }  u_j \rangle $  and $ \e(u_j) > t $
yield  $\h( x_G  ) \ge 1 $.  
Now suppose $ \e(x) = t $   and \eqref{eq.zmir}.  
Then 
 $\h(x ) = 0 $  and  $\h( x_G  ) \ge 1 $ 
imply $\h( x_L + x_i  ) = 0 $.  
From $ f^{t-1}  x_L  = 0 $ and 
$ \h( f^{t-1}   x) = t-1 $  we obtain  $ \h(x_i) = 0 $, 
that is, $ x_i  = u_i + v $, $ v \in   \langle    f  u _i   \rangle $.
It follows from \eqref{eq.ngpl} that $[u_i] $ is a linear manifold of the
form \eqref{eq.orb}. 
\end{proof}

In the course of our paper we shall frequently 
illustrate our results
by a running example.  For that purpose we always 
 use a  vector space $V$ of dimension $10$ and an endormorphism
 $f$  of $V$
with elementary divisors $ \lambda, \lambda^3,  \lambda^6$
such that 
\beq  \label{eq.vedr} 
V = \langle u_1 \rangle \oplus  \langle u_2 \rangle \oplus   \langle u_3 \rangle
\;\; \text{and} \;\;  \big( \!  \e(u_1) ,  \e(u_2)  ,  \e(u_3) \big)  = ( 1, 3, 6) .
\eeq 
 In the following example we 
apply Lemma~\ref{la.nurpd} to determine the characteristic hull of
subspaces. 

\medskip 

\begin{example} \label{ex.btw}
{\rm{
Let  $(V,f) $ be given by \eqref{eq.vedr}. 
We consider two subspaces, namely  
\beq  \label{eq.dex} 
G= \langle z \rangle^c \quad {\rm{with}} \quad  z = u_1 +f u_2 + f^2 u_3,
\eeq
and 
\beq   \label{eq.dey} 
F =  \langle w_1, w_2 \rangle ^c
 \quad {\rm{with}} \quad 
w_1 = u_1 + fu_2,  \:w_2 = fu_2 + f^2 u_3.
\eeq 
We have 
\begin{gather}  \label{eq.gth} 
\begin{aligned} 
  \left[u_1\right] & = u_1
+   \langle u_2  \rangle [ f] +  \langle  u_3  \rangle [ f]  = u_1 +   \langle f^2  u_2  \rangle
 + 
\langle  f^5  u_3  \rangle , 
\\
 [ u_2]    &    = u_2 +   \langle  fu_2  \rangle  +   \langle u_1  \rangle  +  
\langle  u_3  \rangle [ f^3] 
=
u_2 +   \langle  fu_2  \rangle  +   \langle u_1  \rangle  +  \langle f^3  u_3  \rangle
, 
\\ 
[ u_3] &    =
u_3 +   \langle f  u_3  \rangle
+ 
  \langle u_1  \rangle  +   \langle u_2  \rangle, 
\end{aligned} 
\end{gather} 
and     
\begin{gather}  \label{eq.agl} 
\begin{aligned}   
\left[ f u _2  \right]  &  =  f  u_2  +  \langle  f ^2 u_2  \rangle  
+   \langle f^4  u_3  \rangle ,  
\\
 [  f^2 u_3   ]   &  =  f^ 2  u_3 +  \langle f^3  u_3  \rangle
+ 
 \langle  f^ 2  u_2  \rangle .
\end{aligned} 
\end{gather}  
Hence 
\,$
[ z]  =  z + \langle  f ^2 u_2 ,  f^3  u_3   \rangle $\, 
and 
\[
G = \langle z \rangle^c  = \langle z,   f ^2 u_2 ,  f^3  u_3   \rangle . 
\]
From \eqref{eq.gth}  and  \eqref{eq.agl}   we obtain
\[
\left[ w_1 \right] = 
 u_1 +   \langle f^2  u_2  \rangle  + \langle  f^5  u_3  \rangle  +  f u_2  + 
  \langle  f ^2 u_2  \rangle  
+   \langle f^4  u_3  \rangle
= w_1 +   \langle   f^2  u_2  ,   f^4  u_3  \rangle  
\]
and 
\begin{multline*} 
\left[ w_2\right]   =  
 f   u_2 + 
 \langle  f ^2 u_2  \rangle  
+   \langle f^4  u_3  \rangle +  f^ 2  u_3 +  \langle f^3  u_3  \rangle
+ 
 \langle  f^ 2  u_2  \rangle   = 
\\  w_2  + \langle  f ^2 u_2  \rangle   +  \langle f^3  u_3  \rangle , 
 \end{multline*}  
\[
F = 
\langle w_1, w_2 \rangle ^c = \langle w_1, w_ 2  ,   f^2  u_2  ,   f^3  u_3  \rangle 
= \langle w_1, w_ 2   \rangle  . 
\]
Let $ Q \in \{G, F\} $.  Then 
$ Q \cap    \langle u_1  \rangle  = 0 $, $ Q \cap    \langle u_2  \rangle  = \langle f^2 u_2  \rangle $,
and 
$ Q \cap    \langle u_3  \rangle  =  \langle   f^3  u_3   \rangle $.
Thus
\beq \label{eq.xyhr} 
G_H = F_H =    \langle  f^2 u_2 ,    f^3  u_3      \rangle  = W (1, 2, 3) . 
\eeq 
We have  $\pi _1 z  = u_1 \notin G $.  Therefore  $ \pi_1 \in  {\rm{End}}(V,f) $
implies that  the characteristic subspace 
$G $ is not hyperinvariant.
Similarly  we conclude from $ \pi _1 w_1 = u_1  \notin F$  that
 $F $ is   not hyperinvariant. 
}}
\end{example}



\section{Frames}  \label{sbs.frm} 

In  this  section we  consider the 
hyperinvariant frame  $(X_H, X^h) $ 
 of  a characteristic subspace $X$ and we 
describe the connection between 
 $ X_H$   and  $X^h $.

\begin{theorem} \label{thm.hyhl} 
Let  $X $ be a characteristic subspace of $V$ and let the
numbers $ r_j, \mu_j $, $ j = 1, \dots , m $, be given by 
\beq 
\label{eq.ngws}  
\langle u_j \rangle \cap X   =       \langle f^{r_j} u_j \rangle   \:  \: 
\mbox{and} \: \:  \:     \pi_j X = \langle f^{\mu_j}u_j \rangle  ,    \: \: 
 j = 1, \dots, m   .
\eeq 
\begin{itemize}
\item[{\rm{(i)}}] 
Then 
\beq \label{eq.mrjo} 
\mu_j  = 
 \begin{cases} 
r_j   &  \mbox{if}   \quad   \langle u_j\rangle \cap X   =   \pi_jX 
\\
r_j- 1 &  \mbox{if}  \quad  \langle u_j \rangle \cap X    \subsetneqq  \pi_jX  .
\end{cases} 
\eeq
\item[{\rm{(ii)}}] 
The hyperinvariant hull of $X$ is  
\beq \label{eq.nrpix}
X ^h    =  \sum \nolimits _{j = 1} ^m  \pi _j X 
=  \langle   f^{\mu_1} u_1 , \dots,   f^{\mu_m} u_m \rangle . 
\eeq  
\end{itemize} 
\end{theorem}

\begin{proof} 
(i) 
If    $ X \cap \langle u_j\rangle =\pi_j X  $ 
then  $\mu_j = r_j$. 
Now suppose  $ X \cap \langle u_j\rangle  \subsetneqq \pi_j X $.
Then 
$    \langle f^{r_j}u_j \rangle     \subsetneqq \langle     f^{\mu_j}u_j \rangle  $
implies  $ r_j  > \mu _j $, and therefore $ r_j \ge 1$.
Because of $  f^{\mu_j}u_j  \in \pi_j X $ we 
can  choose   an element $x \in X $ such that  
\[
 x = \sum \nolimits _{i=1} ^m   x _i , \:\;   
x_i \in \langle u_ i \rangle, \:  i=1, \dots, m,  \quad {\rm{and}} \quad 
  x_j =  \pi_j x = f^{\mu_j}u_j .
\] 
Let $\alpha = \alpha ( u_j , u_j + fu_j )$.  
Then  $\alpha x = x + f^{\mu_j+1}u_j$.
By assumption  the subspace  $X$ is characteristic. 
Therefore  $ \alpha x \in X$.
 Hence $f^{\mu_j+1} u_j\in X \cap \langle u_j \rangle = \langle f^{r_j}u_j \rangle$.
This  implies $\mu_j +1 \geq r_j$, and yields   $\mu_j = r_j -1$,
and completes the proof of  \eqref{eq.mrjo}.

(ii) 
Set 
$   \tilde X = \sum  _{j=1} ^m \, \pi _j X$. 
Then   $ X \subseteq   \tilde  X $ 
and  $ \pi _j  \tilde X =  \pi _j  X  $, $j = 1, \dots , m $.
Hence 
\beq \label{eq,stwu}
  \tilde X = \sum \nolimits  _{j=1} ^m \, \pi _j  \tilde X =  
\sum  \nolimits   _{j=1} ^m \,   \big \langle   f^ {\mu _j } u_j  \big\rangle  .
\eeq  
Let us show that $  \tilde X  $ is hyperinvariant. 
We first prove that $  \tilde X  $ is characteristic.
We consider the generators $  f^{\mu_j } u_j $ of $  \tilde X $. 
Let $\alpha \in {\rm Aut}(V,f)$. 
If  $ \langle u_j\rangle \cap X   =   \pi_jX$  
then  
\[
 \langle  f^{\mu_j}u_j \rangle =  \langle  f^{r_j}u_j \rangle  = 
  \langle u_j\rangle \cap X
 \subseteq X ,\]  
and therefore  $\alpha( f^{\mu_j}u_j  )\in X \subseteq  \tilde{X}$.
If  $   \langle u_j\rangle \cap X   \subsetneqq  \pi_jX$, 
then 
$ u _j $  is unrepeated. 
Hence  $\alpha u_j = u_j + y$ with
$y \in  \langle u_1, \dots, u_{j-1} , f u_j,  u_{j+1} , \dots , u_m \rangle $,
and   $  u_j + y   $ is a generator with $ \e( u_j + y ) =  \e( u_j) $. 
Then
\beq \label{eq.smjy} 
\alpha ( f^{\mu_j}  u_j )  = f^{\mu_j}u_j + f^{\mu_j}y. 
\eeq 
Let us show that $f^{\mu_j}y \in \tilde{X}$.  
From \eqref{eq,stwu} follows  $  f^{\mu_j}u_j   \in  \pi _j \tilde X  =  \pi _j X  $.
Hence  $  f^{\mu_j}u_j  = \pi _j x $ for some $x \in X$. 
Then  
$ x = x_1 + \cdots + x_m$  with $x_i\in \langle u_i \rangle, i=1, \dots, m$,  and 
$x_j = f^{\mu_j}u_j$. 
Let
 $\beta \in {\rm Aut}(V,f)$  be the automorphism  that exchanges 
$ u_j $  by $ u_j + y $.  
Then  $\beta x = x + f^{\mu_j}y \in X$, which implies $ f^{\mu_j}y \in X \subseteq \tilde{X}$. 
We have 
 $  f^{\mu_j}u_j \in \tilde X $.     
Hence 
\eqref{eq.smjy} 
yields  $\alpha( f^{\mu_j}u_j ) \in \tilde{X} $.
Thus we have shown that  $  \tilde X  $ is  characteristic.
It follows from  Theorem~\ref{thm.wogn}(ii)   that $  \tilde X  $ is 
hyperinvariant. 
Then $ X \subseteq \tilde{X}  $  implies 
$  X^h  \subseteq \tilde{X}  ^h  =  \tilde{X} $. 
On the other hand 
we obtain  $ \tilde{X}  \subseteq X^h $, since  $  X^h $ is the hyperinvariant hull of $X$.
Therefore  $  \tilde{X}   = X^h$.
\end{proof}

Let  \,$ \vec e_1 = (1, 0 , \dots , 0 ) ,  \; \dots,  \vec e_m = (0 , \dots , 0, 1 ) $\,
be the row vectors of size~$m$. 
We combine the preceding theorem with Theorem~\ref{thm.fhl}  and 
  Theorem~\ref{thm.wogn}.

\begin{corollary} \label{cor.vlrg} 
Let $X $ be a characteristic subspace of \,$V$ 
and let the integers $r_j$, $\mu _j$, $ j=1, \dots , m $ 
be defined by \eqref{eq.ngws}. 
Set $\vec r = (r_1, \dots , r_m) $ and $\vec \mu =  (\mu_1, \dots , \mu_m)$. 
\begin{itemize}
\item[{\rm{(i)}}] 
The hyperinvariant frame $ (X_H, X^h ) $ of $X$  consists of 
\beq \label{eq.njrp}
X _H     =  \sum \nolimits _{j = 1} ^m  \big( \langle u_j  \rangle \cap X   \big) = 
 \langle   f^{r_1} u_1 , \dots,   f^{r_m} u_m \rangle  = W (\vec r) , 
\eeq
and 
\[
X ^h    =  \sum  \nolimits _{j = 1} ^m  \pi _j X 
=  \langle   f^{\mu_1} u_1 , \dots,   f^{\mu_m} u_m \rangle  = W(\vec \mu).
\]
If $J(X) =  \{ i_1 ,  \dots , i_k\}  \subset I_u $,  $i_1 < \cdots < i_k$, 
then  
\beq \label{eq.rplmu} 
\vec r = \vec \mu +  \sum \nolimits  _{s = 1} ^k \vec e_{i_s}
 \quad and  \quad
 \vec r , \vec \mu  \in \mathcal{L}( \vec t \, ) .   
\eeq 
\item[{\rm{(ii)}}] 
Let   \,$  D (X) = {\rm{span}}\{ f^{\mu_i} u_i;     \:  i \in J(X)  \}  $.   
Then \,$ X^h = X_H \oplus D(X)  $.
The subspace $X$ is  hyperinvariant if and only if $D(X)  =  0 $.  
\end{itemize} 
\end{corollary} 

\medskip

We extend  Example~\ref{ex.btw}  taking into account the results of the
the preceding corollary.

\begin{example}  \label{ex.wct} 
{\rm{
Let $(V,f) $ be given by  \eqref{eq.vedr}. 
Let the subspaces 
$G$ and $F$  be defined by  \eqref{eq.dex}    and  \eqref{eq.dey}, respectively. 
Then  
\[ 
 G= \langle  u_1 +f u_2 + f^2 u_3,  f ^2 u_2  ,   f^3 u_3  \rangle 
\quad \text{and} \quad 
F =   
\langle u_1 + fu_2 , fu_2 + f^2 u_3 \rangle 
\]
yield 
 \[ G^h  =     F ^h  =  \langle u_1 ,  f u_2  ,   f^2u_3 \rangle =  W( \vec \mu )
\:  \: \: \text{with} \:  \:  \:\vec \mu = (0, 1, 2)
\] 
and $ \dim W( \vec \mu )   = 7$. 
Recall \eqref{eq.xyhr}, that is 
\[
G_H = F_H =   \langle f^2  u_2,  f^3 u_3 \rangle =   W(\vec r) 
\:\:\: \text{with} \:\:\:  \vec r = (1,2,3) , 
\] 
and $\dim  W(\vec r)  = 4 $. 
Hence $ \big( W( 1,2,3) , W( 0, 1,2 ) \big) $ is the hyperinvariant frame
for both $G$ and $F$.
If  $Q \in \{G, F\} $ then 
\,$ \langle u_j \rangle \cap Q   \subsetneqq   \pi _j Q$,  $j = 1,2,3 $, 
implies $  J(Q) = \{ 1,2,3  \} $.
Then 
\[
\sum \nolimits _{\kappa \in  J(Q) }  \vec e _\kappa = (1,1,1) ,
\]
and  the  relation 
\,$
 \vec r  =  \vec \mu  + \sum \nolimits  _{\kappa \in  J(Q) } \,  \vec e _\kappa $\, 
in \eqref{eq.rplmu} is satisfied. 
Moreover,  $ D (Q) = {\rm{span}}\{ u_1,    f  u_ 2,    f ^2 u_3 \}   $ such that 
$Q^h = Q_H \oplus D$. 
}} 
\end{example}

When is a pair $ \big(  W(\vec r),  W(\vec \mu) \big)  $ 
 the hyperinvariant frame of a subspace
 $X   \in {\rm{Chinv}}(V) \! \setminus \!  {\rm{Hinv}} (V,f)   $?
From Lemma~\ref{la.strug}  
below we obtain necessary conditions 
that involve the set $J(X)  = 
\{ j ; \:    \langle u_j \rangle \cap X   \subsetneqq   \pi _j X \}$. 
The strict inequality \eqref{eq.smpt} below can be interpreted in the
view of Shoda's theorem. 
It follows from  \eqref{eq.smpt}  that 
 a given $f$ can  give rise to a characteristic subspace that
is not hyperinvariant   only if $f$ has  unrepeated elementary divisors 
$ \lambda ^R $ and $ \lambda ^S $ such that the integers 
 $R$ and $S$ are
not consecutive.

\medskip 

\begin{lemma} \label{la.strug}  
Suppose 
$X$ is  characteristic  and not hyperinvariant with 
\[
 X^h  = W(\mu_1, \dots, \mu_m)  \quad \text{and} \quad 
 X_H   = W(r_1, \dots , r_m). 
\] 
Then  \,$  |  J(X) | \ge 2 $ and $\vec \mu $ has the following properties. 
\begin{itemize}
\item[{\rm{(i)}}] 
If  \,$  p    \in J(X) $ then 
\beq \label{eq.ugbl} 
   \mu_ p < \mu_q   \:  \:  \,   \mbox{if}  
\:  \:  \,   p < q .
\eeq 
\item[{\rm{(ii)}}]
If \,$  q   \in J(X) $  then  
$0 <  t_q - \mu_q $ 
and 
\beq \label{eq.ugbl2} 
t_p -   \mu_ p < t_q - \mu_q  \: \: \: \mbox{if} \: \:     
\:  \,   p < q .
\eeq 
\item[{\rm{(iii)}}] 
If  \,$  p, q   \in J(X)  $\,  then  
\beq \label{eq.smpt} 
t_p +1 < t_q  \: \,   
\mbox{if}   \:  \,\,     p < q . 
\eeq
\end{itemize}
\end{lemma} 

\begin{proof} 
We show that  \,$  |  J(X) | > 1 $.
Suppose $X$ is characteristic and 
\beq \label{eq.fals}
 \langle u_j \rangle \cap X  =   \pi _j X  \quad \text{for all}  \quad 
j \in \{1, \dots , m \}\! \setminus \! \{s\} .
\eeq 
Let $x \in X $ be written as 
$ x = \sum _{j = 1}^{m} x_j $, $ x_j = \pi _j x $, $ j = 1, \dots , m$. 
Then \eqref{eq.fals} implies $x_s \in X$. Hence $  \langle u_s \rangle \cap X  =   \pi _s X $.
Therefore $X$ is hyperinvariant (by Theorem~\ref{thm.wogn}).

(i)
Suppose  \,$p \in J(X) $ and 
$   p < q $.  Since $ u_p $ is unrepeated
we have $ t_p < t_q $.  
From  $\e(u_p) < \e (u_q) $  follows  $ \e(u_q + u_p ) = \e(u_q) $.
Let $ \alpha = \alpha ( u_q, u_q + u_p)$. Then $ \alpha X \subseteq X$. From  
\[
f^{\mu_q}  u_q   \in W(\vec  \mu ) = X^h =  \sum \nolimits  _{j=1} ^m  \pi_j X  
\]
follows $  f^{\mu_q}  u_q \in \pi _q X $.  Therefore $ f^{\mu_q} u_q  = \pi _q x  $
for some  $ x \in X $.  
Then $ \alpha x = x +  f^{\mu_q} u_p  \in X$, and therefore
$  f^{\mu_q } u_p  \in X $.  
Hence 
\[
 f^{\mu_q } u_p  \in  X  \cap \langle u_p \rangle =  \langle f^{r_p} 
 u_p \rangle = \langle f^{\mu_p +  1}   u_p \rangle ,
\]
which implies $\mu_q \ge \mu_p + 1 $. 

 (ii) 
If  $ 0  = t_i -  \mu_i $ then 
$ \pi _i X =  \langle f^{\mu _i } u _i  \rangle  = 0 $. Hence 
 \,$ 0 =   \langle u _i \rangle  \cap X =  \pi _i X $, and therefore
$ i \notin J(X) $.   
Suppose  $ q \in J(X) $ and $ p < q $. 
Then $ t_p < t_q $, and therefore  
\[
\e(  u_p + f^{t_q - t_p } u_q ) = 
\e(u_p ) = t_p . 
\]
Let $ \alpha = \alpha( u_p, u_p + f^{ t_q - t_p  } u_q ) $.
Because of  $f^{\mu_p} u_p \in \pi _p X $
there exists an $x \in X $ such that  such that $ \pi_p x =  f^{\mu_p} u_p $.
Then
\,$
\alpha  x = x +  f^{\mu_p  +  t_q - t_p } u _q  \in X $. 
Hence 
\[
 f^{\mu_p  +  t_q - t_p } u _q  \in X \cap \langle u_q \rangle = 
\langle f^{r_q}  u_q \rangle =  \langle f^{\mu_q + 1}  u_q \rangle . 
\] 
Therefore $ \mu_p  + t_q - t_p   \ge \mu_q + 1  $, which implies 
$ t_q - \mu_q > t_p - \mu _p $. 

(iii)  Suppose  
$  p, q   \in   J(X) $, 
$p < q $. 
Then 
\eqref{eq.ugbl}  and  \eqref{eq.ugbl2} imply  
\,$
1 \,  \le \,   \mu_q -  \mu_ p \,  < \,   t_q - t_p  $. 
Hence  $ t_ p + 1 <   t_q $.
\end{proof}

\medskip 

From Theorem~\ref{thm.hyhl} and   Lemma~\ref{la.strug} we obtain the following. 

\begin{theorem} \label{thm.wzmn}
Let  $ \vec \mu ,  \vec r  \in \mathcal{L}( \vec t \, )$ and 
let 
\beq \label{eq.sjbte}  
J = \{ i_1 ,  \dots , i_k\}  \subset I_u ,    \,    \, 
i_1 < \cdots < i_k,     \,  \,  |J | \ge 2 .
\eeq 
If  $X$ is  characteristic  and not hyperinvariant and   $ J(X) = J $ and 
$ X^h  = W(\vec \mu)$ and $ X_H   = W( \vec r )$,  
then
\[
\vec r = \vec \mu +  \sum \nolimits  _{s = 1} ^k \vec e_{i_s} ,
\]
and  $\vec \mu_J =  (  \mu _{i_1} , \dots, \mu _{i_k} ) $ satisfies 
\beq \label{eq.krnap} 
0 \le  \mu _{i_1}   <  \dots < \mu _{i_k} 
  \quad \textrm{and} \quad
0 < t_{i_1}  -  \mu _{i_1}    < \dots <  t_{i_k} -  \mu _{i_k} ,
\eeq
and 
\beq  \label{eq.smtrpi} 
t_{i_s} +1 < t_{i_{(s +1)}} ,  \: 
   s = 1, \dots, k-1.
\eeq
\end{theorem}

\medskip

Assuming \eqref{eq.sjbte} we prove in Section~\ref{sbs.jso}
 a converse of the preceding theorem.
If the entries of  $\vec t_J = ( t_{i_1} , \dots , t_{i_k} ) $ satisfy 
\eqref{eq.smtrpi} then there exist tuples 
$\vec \mu _J = ( \mu _{i_1} ,  \dots , \mu_{i_k}) $
 of nonnegative integers such that the inequalities 
 \eqref{eq.krnap}  hold.  
One can check that   $\vec \mu _J $  satisfies  \eqref{eq.krnap} 
if and only if 
  $ 0 \leq \mu_{i_1} < t_{i_1} $
and   
 $ \mu_{i_{s+1}} = \mu_{i_s} + \delta_s $
with  
  $  1 \leq \delta_s < t_{i_{s+1}} - t_{i_s} $, 
 $ s =1, \dots, k-1 $. 
In Lemma~\ref{la.njfos} we shall  see that  one can 
 extend  such a 
$\vec \mu _J$  
to an $m$-tuple $\vec  \mu  $
such that $ \vec  \mu \in \mathcal{L}( \vec t \,  ) $ and  
$ \vec \mu +    \sum \nolimits  _{j \in J}  \vec e_j  \in \mathcal{L}( \vec t \,  ) $.
Then, using Theorem~\ref{thm.main}  
one can construct a characteristic non-hyperinvariant subspace 
$X$ such that $J(X) = J$,
and 
\[ 
X_H \cap  \langle u_j \rangle =  \langle f^{\mu_j +1}  u_j \rangle
\quad \text{and} \quad 
X^h   \cap  \langle u_j \rangle  =  \langle f^{\mu_j}  u_j \rangle, \, j \in J . 
\]

\section{Intervals}  \label{sbs.int} 
Let 
$ A , B \in  {\rm{Inv}}(V, f) $ and  $A \subseteq B $. 
  The  set  
\[
[A, B ] = \{C \in  {\rm{Inv}}(V, f) , \,  A \subseteq C  \subseteq B \} 
\]
is  an interval    of the invariant subspace  lattice  $  {\rm{Inv}}(V,f) $. 
In this section we study intervals of the form 
$[ X_H , X^h  ]$,  which  can arise from  subspaces 
$X   \in {\rm{Chinv}}(V, f) \! \setminus \!  {\rm{Hinv}} (V,f)   $.  
A useful property of direct sums and intervals is the following. 
\begin{lemma} {\rm{\cite[p.\  38]{FuI}}} \label{la.mod} 
Let $ A, B, C , D  $  be subspaces of  \,$V$. 
Suppose   
$ B = A \oplus D $ and   $C \in [A, B]$.
Then   $ Z = C \cap D $
is the unique subspace satisfying  
\beq \label{eq.zwbs} 
  Z \subseteq D \quad \text{and} \quad C = A \oplus Z  .
\eeq 
\end{lemma} 

\begin{proof} 
The modular law  implies 
\,$ C = B \cap C = ( A \oplus D  ) \cap C =  A \oplus ( D   \cap C ) $.  
Hence $ Z = D   \cap C  $ has the properties \eqref{eq.zwbs}. 
Conversely, if  \eqref{eq.zwbs} holds,  then 
\,$ C \cap D  = ( A \oplus Z ) \cap D =  ( A \cap D) \oplus Z = Z $. 
\end{proof}

For the proof  of Theorem~\ref{thm.main} 
we need the following  auxiliary result. 

\begin{lemma} \label{la.nrlm} 
Let $J  = \{i_1, \dots , i_k\}  \subseteq I_u$,    $ i_1 < \cdots   <  i_k$,
$ 2 \le k  $. 
Suppose 
 $ \vec  \mu   = ( \mu_1, \dots , \mu_m )   \in \mathcal{L}( \vec t \, )  $ and 
\beq   \label{eq.glu1}
0 \le  \mu _{i_1}  < \mu _{i_2}   < \dots < \mu _{i_k} 
\eeq 
and 
\beq   \label{eq.glu2}
0 < t_{i_1}  -  \mu _{i_1}  <  t_{i_2}  - \mu _{i_2}    < \dots <  t_{i_k} -  \mu _{i_k} ,
\eeq 
and suppose 
$ \vec r =   
\vec \mu +   \sum  \nolimits  _{s = 1} ^k  \vec e_{i_s }  \in  \mathcal{L}( \vec t \, )  $.
Let $ U = (u_1, \ldots , u_m ) \in \mathcal{U} $ and 
 $ \alpha \in  {\rm{Aut}}(V,f) $. 
\begin{itemize}
\item[{\rm{(i)}}]
Then 
\beq \label{eq.fwga}  
   \alpha f^{\mu_{i_s}}    u_{i_s}   = f^{\mu_{i_s}}    u_{i_s}   +   w_{i_s} 
\quad with \quad  w_{i_s}  \in W(\vec r) , \: s= 1, \dots , k.
\eeq  
\item[{\rm{(ii)}}]
If  $z \in 
{\rm{span}}\{ f^{\mu_{i_1}  }  u_{i_1} , \dots,  f^{\mu_{i_k}  }  u_{i_k} \}  $
then  
$\alpha z = z + w $  with $ w \in  W(\vec r) $. 
\end{itemize}
\end{lemma}

\begin{proof} 
(i)  The generator $  u_{i_s}  $ is unrepeated. 
Therefore Lemma~\ref{la.nurpd} yields   
\begin{multline*} 
  \alpha   u_{i_s} =  u_{i_s}    +   v_{i_s} + y_{i_s}  
\quad
{\rm{with}}   \quad
   v_{i_s} \in   \langle   f  u _{i_s  }  \rangle ,  \quad
{\rm{and}}   \\ y_{i_s}  \in 
  \langle  u_j  ;  \; j = 1, \dots , m ; \;   j \ne i_s  \rangle ,
  \;\:  \e( y_{i_s} ) \le \e ( u_{i_s}  ) = t_{i_s} , 
\,\, 
s = 1 , \dots , k. 
\end{multline*} 
Then 
\[
  \alpha f^{\mu_{i_s}}    u_{i_s} =  f^{\mu_{i_s}} u_{i_s}  + w_{i_s} 
\quad \text{with} \quad  w_{i_s}  =  f^{\mu_{i_s}}    v_{i_s}  +  f^{\mu_{i_s}} y_{i_s} .
\]
We have \,$  \alpha f^{\mu_{i_s}}    u_{i_s}     \in W( \vec  \mu ) $, 
since 
\,$ f^{\mu_{i_s}}    u_{i_s}    \in  W( \vec  \mu ) $\,
and  $ W( \vec  \mu ) $   is hyperinvariant. 
Moreover 
$  f^{\mu_{i_s}}    v_{i_s} \in \langle   f^{\mu_{i_s +1}} u_{i_s}   \rangle \subseteq 
 W( \vec r)  \subseteq W(\vec  \mu) $.
Hence 
\beq \label{eq.inpwo}  
  f^{\mu_{i_s}} y_{i_s}   \in W( \vec  \mu )  
= \langle  f ^{\mu_1} u_1 , \dots,   f ^{\mu_m} u_m   \rangle.
\eeq 
It remains to show that 
$ f^{\mu_{i_s}} y_{i_s}   \in W( \vec r ) $.  
Let   $y_{i_s} $ be written as 
\[
 y_{i_s}  = 
 \sum \nolimits _{ j = 1, j \ne i_s} ^m  x_j  \quad \text{with} \quad    x_j \in 
  \langle  u_j  \rangle .
\] 
Then  \eqref{eq.inpwo} implies 
\beq \label{eq.vlrb} 
 f^{\mu_{i_s}} x_j  \in  W( \vec  \mu ) \cap \langle  u_j  \rangle =
\langle  f^{\mu_j }   u_j \rangle .
\eeq
If   $ j \notin J =  \{i_1, \dots , i_k \} $  
then 
$ r_j = \mu _j$,
and  \eqref {eq.vlrb}  yields 
\,$ f^{\mu_{i_s}} x_j  \in W(\vec r ) $.   Suppose $ j \in J$ and  $  j > i_s $.
Then  \,$\e( x_j ) \le \e(y_{i_s} ) \le  t_{i_s} $\,  implies  
\,$ x_j \in f^{ t _j -  t_{i_s} }  \langle  u_j  \rangle $.
Hence it follows from 
\eqref{eq.glu2}  that 
\[
 f^{ \mu_{i_s}}  x_j   \in   f^{ t _j -  t_{i_s}  + \mu_{i_s} }  \langle  u_j  \rangle    =
 f^{   \mu _j  + (  t _j -    \mu _j  ) -   ( t_{i_s}  - \mu_{i_s} )  }  \langle  u_j  \rangle  
\subseteq   f^{   \mu _j   + 1 }  \langle  u_j  \rangle = f^{  r_ j  }  \langle  u_j  \rangle ,
\]
and   we see in this case  
that   $  f^{ \mu_{i_s}}  x_j \in W(\vec r ) $. 
Now suppose 
$ j \in J $ and $ i_s >  j$.  If  $ j = i_\tau$, $ \tau < s$ 
 then \eqref{eq.glu1} implies  $ \mu_{i_s} >   \mu _{i_\tau} =  \mu_j $,  
and we obtain 
  \[
 f^{ \mu_{i_s}}  x_j  \in  f^{ \mu_{i_s}}  \langle  u_j  \rangle \subseteq 
 f^{ \mu_j  + 1 }  \langle  u_j  \rangle =  f^{ r_j}  \langle  u_j  \rangle ,
\]
and  therefore  $  f^{ \mu_{i_s}}  x_j \in W(\vec r ) $. 
Hence  $  f^{ \mu_{i_s}}  y_{i_s} \in W(\vec r ) $,  which 
completes the proof of \eqref{eq.fwga}. 

(ii) Let 
\,$ z = \sum \nolimits _{s = 1} ^k c_s  f^{ \mu_{i_s}} u_{i_s}$, $c_s \in K $. 
Then \eqref{eq.fwga} implies 
\[
\alpha z = z +   \sum \nolimits _{s = 1} ^k c_s w_{i_s}  \in z + W( \vec r ) . 
\]
\end{proof}

\medskip 

We have seen in  Theorem~\ref{thm.wzmn} 
that a subspace $X   \in {\rm{Chinv}}(V, f) \! \setminus \!  {\rm{Hinv}} (V,f)   $
with  $ X _H = W(\vec r ) $, $X^h = W(\vec \mu)  $  and $J(X ) =  \{i_s\} _{s =1} ^k $
satisfies the conditions \eqref{eq.smip}  - \eqref{eq.jnul}
   of Theorem~\ref{thm.main}  below. 
Hence, if $ (X_H , X^h) $ is the hyper\-invariant frame of $X$ then
the following theorem describes the corresponding interval 
$[  X_H , X^h ] $.

\begin{theorem}   \label{thm.main} 
Let 
$J  = \{i_1, \dots , i_k\} \subseteq I_u $, $ i_1 < \cdots < i_k$, 
$ 2  \le k $. 
Assume  
\beq  \label{eq.smip}  
t_{i_s} +1 < t_{i_{(s +1)}} ,  \: 
   s = 1, \dots, k-1.  
\eeq
Let 
\,$ \vec  \mu  ,  \vec r   \,   \in \,  \mathcal{L}( \vec t \, ) $\,  be such that  
\beq \label{eq.gluv}  
0 \le  \mu _{i_1} < \dots < \mu _{i_k} 
  \quad \textrm{and} \quad
0 < t_{i_1}  -  \mu _{i_1}  < \dots <  t_{i_k} -  \mu _{i_k}
\eeq
and 
\beq \label{eq.jnul} 
 \vec r = \vec \mu + \sum  \nolimits _{s = 1 } ^k  \vec e_{i_s} 
\eeq 
hold.    Set  
 \beq \label{eq.dmjar} 
 D _{\vec \mu _J}  
 = {\rm{span}}\{ f^{\mu_{i_1}  }  u_{i_1} , \dots,  f^{\mu_{i_k}  }  u_{i_k} \} .
\eeq 
\begin{itemize}
\item[{\rm{(i)}}]
Each subspace     $ X  \in  [ W(\vec r), W(\vec \mu) ] $ 
is characteristic. 
Moreover, 
$ X  \in    [ W(\vec r), W(\vec \mu) ] $ if and only if 
$X =  W(\vec r) \oplus Z $  for some subspace $ Z \subseteq   
  D _{\vec \mu _J}  $. 
\item[{\rm{(ii)}}] 
A subspace $ X   \in   [ W(\vec r), W(\vec \mu) ] $
 is hyperinvariant
if and only if 
\beq \label{eq.mgto}   
X = W(\vec r) \oplus {\rm{span}} \{ f^{\mu_{\tau_1}  }  u_{\tau_1} ,
\dots,  f^{\mu_{\tau_q}  }  u_{\tau _q} \} 
\eeq 
for some subset $ T = \{\tau_1, \dots , \tau _q \} $ of $J$. 
\end{itemize} 
\end{theorem}

\begin{proof} 
(i)  Corollary~\ref{cor.vlrg}(ii)    
implies 
$  W(\vec \mu)  = W(\vec r) \oplus  D _{\vec \mu _J}  $.
If 
\beq \label{eq.sbwt} 
 W(\vec r) \subseteq X  \subseteq W(\vec \mu) 
\eeq 
then 
 the subspace $ Z =   X \cap    D_{\vec \mu _J} $ satisfies $X =  W(\vec r) \oplus Z $ 
(by  Lemma~\ref{la.mod}).
Let $x \in X $. Then $ x = y + z  $ with $ y \in  W(\vec r) $, $ z \in Z $.
If $\alpha \in {\rm{Aut}}(f,V) $ then 
Lemma~\ref{la.nrlm}   implies 
$ \alpha z = w + z  $, $w \in  W(\vec r) $.
Since $  W(\vec r)  $ is hyperinvariant we have \,$ \alpha  y    \in  W(\vec r)  $, 
and 
 we obtain 
\,$ \alpha x \in W(\vec r) \oplus Z = X  $.

(ii)   A subspace $X$ is hyperinvariant
and satisfies  \eqref{eq.sbwt} 
 if and only if  
$ X = W(\vec \eta) $ for some  $\vec \eta \in  \mathcal{L}( \vec t \, ) $
with $ \vec \mu   \preceq \vec \eta  \preceq  \vec r$, 
that is, if and only if  $ \vec  \eta  = \vec \mu +  \sum _{\nu \in T } \vec e_{\nu} $
for some subset $T$ of~$J$. 
 \end{proof} 

\medskip

In \cite{Ba}  a subspace $Y$  is called a minext subspace if
it  complements  
a hyperinvariant subspace  $W$ such that $X = W \oplus Y $
is characteristic   non-hyperinvariant and $ X_H = W$.

\medskip 

\begin{example}  \label{ex.spcd} 
{\rm{Let $(V,f) $ be given by  \eqref{eq.vedr}.  
Then  $\vec t = (1, 3, 6)$ and 
  $ I_u = \{1,2,3\}$. 
The  sets $J$ with property \eqref{eq.smip} are 
$ J = \{1,2,3\} $, $ J = \{1,2\} $,  $ J = \{1,3\} $,  $ J = \{2,3\} $.
 In the following we  consider $ J = \{1,2,3\} $ and 
$ J =  \{1,3\}  $. 

Case $ J = \{1,2,3\} $.   Then $\sum _{j \in J} \vec e_j = (1,1,1) $.  
If \,$ \vec \mu = (0,1,2) $,  
\,$  \vec r = (1, 2, 3) $, 
or \,$ \vec \mu = (0, 1,3) $, \,$  \vec r = (1, 2, 4) $
then   $ \vec \mu , \vec r \in  \mathcal{L}( \vec t \, ) $
holds and \eqref{eq.gluv}  is satisfied. 
Let us  consider in more detail the case 
\,$ \vec \mu = (0,1,2) $, \,$  \vec r = (1, 2, 3)$.
We have  
\[
W( \vec \mu ) = \langle u_1,  f u_2, f^2 u_3 \rangle , \:   \:  
  W( \vec r ) = \langle  f^2 u_2, f^3 u_3 \rangle 
\]
and 
\,$
D_{\vec \mu} =    {\rm{span}} \{ u_1,  f u_2, f^2 u_3\} $. 
It is well known (see \cite{Kn}, \cite{NSW}) that  
the number of $k$-dimensional subspaces of an $n$-dimensional 
vector space over the field $\operatorname{GF}(q) $ is equal to  
the $q$-binomial coefficient 
\[
\binom{n}{k}_{ \!q}   =  \frac{( q^{n} - 1) (   q^{n-1} - 1 )  \cdots ( q^{ n-k +1} -1) } 
{( q^k - 1) ( q^{k-1}  - 1)  \cdots ( q - 1 ) }  . 
\]
Hence the vector space   $D_{\vec \mu} $ has
\[
  \binom{3}{0} _2  +  \binom{3}{1} _2 +  \binom{3}{2} _2 +  \binom{3}{3} _2
=
 1  + 7 + 7  + 1 = 16
\] 
 subspaces,  and therefore the interval  $[ W( \vec r ) ,  W( \vec \mu ) ] $ contains 
$ 16 $ characteristic subspaces.
We have $2^3=8$  choices for a subset $T$ of $J$. Hence there are
$8$ hyperinvariant subspaces in $[ W( \vec r ) ,  W( \vec \mu ) ]$, 
e.g. 
\beq \label{eq.lpmr}
 W(1, 2, 3 ) + \langle u_1 \rangle   = W(0,2,3) \quad \text{and}  \quad 
 W(1, 2, 3 )  +  \langle  fu_2,  f^2 u_3  \rangle  =  W(1,1, 2) .
\eeq 
Thus there are $8$  subspaces in $[ W( \vec r ) ,  W( \vec \mu ) ]$
that are not hyperinvariant.  
Examples  of such   subspaces  are 
\[ 
Y_2 = 
W( \vec r ) \oplus    {\rm{span}} \{   u_1 + f u_2 \} = 
\langle  f^2 u_2, f^3 u_3 ,   u_1 + f u_2 \rangle   = \langle     u_1 + f u_2  \rangle ^c 
\]
with  \,$ \dim Y_2  = 5 $,  
and 
\[
Y_3 = W( \vec r ) \oplus    {\rm{span}} \{   u_1 +  f u_2 , f^2 u_3 \}  
= \langle u_1 +  f u_2 ,  f^2 u_3     \rangle  =  \langle u_1 +  f u_2 ,  f^2 u_3  
   \rangle ^c 
\]
with   \,$ \dim Y_3 = 6 $.  Moreover, we know from 
 Example~\ref{ex.btw} and Example~\ref{ex.wct} 
that the  subspaces 
\beq \label{eq.hnbst} 
G = \langle u_1 +f u_2 + f^2 u_3  \rangle^c  =  
  W(1, 2, 3 ) \oplus     {\rm{span}} \{ u_1 +f u_2 + f^2 u_3  \}
\eeq 
and 
\beq \label{eq.rsmt}
 F =  \langle u_1 + fu_2 , fu_2 + f^2 u_3  \rangle ^c =   W(1, 2, 3 )  
\oplus       {\rm{span}} \{ u_1 + fu_2 ,   fu_2 + f^2 u_3 \} 
\eeq 
are not hyperinvariant.  

Case $J= \{1, 3\}$. 
 There are four pairs $ (\mu_1, \mu_3)$ that  satisfy \eqref{eq.gluv}, 
namely \[
 (\mu_1, \mu_3) \in \{
(0,1), (0,2), (0,3), (0,4) \} .
\]
We focus on  $(\mu_1, \mu_3) =  (0,2) $. 
Then $ \vec \mu = (\mu_1,  \mu _2 , \mu_3)   \in \mathcal{L}( \vec t \, ) $
if  
\[  \vec \mu \in \{    (0, 0  , 2) ,   (0, 1  , 2)  ,  (0, 2  , 2)  \} , \]
and we have \,$ \vec  r =  \vec \mu + (1 , 0 , 1)  \in \mathcal{L}( \vec t \, ) $\,
if   and only if $ \vec \mu =  (0, 1  , 2) $ or $ \vec \mu =  (0, 2  , 2)  $.
Then $ \vec r  =  (1, 1  , 3) $ or $ \vec r  =  (1, 2  , 3)$, respectively,  and 
$ D_{\vec \mu _J} =    {\rm{span}} \{ u_1, f^2 u_3 \} $ with 
\mbox{$\dim  D_{\vec \mu _J} = 2$}.
Hence,  besides their  endpoints   the respective   intervals 
$ [  W ( \vec r  ) , W ( \vec \mu  ) ] $
contain two subspaces which are hyperinvariant, namely    
\[
W(\vec r ) + {\rm{span}} \{ u_1 \}   =  W ( \vec r - \vec e_1) 
\quad  \text{and} \quad  W(\vec r ) + {\rm{span}} \{ f^2 u_3 \} = W(\vec r -  \vec e_3)  ,
\]
together with the non-hyperinvariant  subspace 
\[   W(\vec r )  + {\rm{span}} \{ z\} , \: z=   u_1+ f^2 u_3 .
\]
%
In the case 
$ ( \vec \mu,  \vec r  ) = \big(  (0, 1  , 2),   (1, 1  , 3) \big) $ 
the elements of $ [  W ( \vec r  ) , W ( \vec \mu  ) ] $ are 
$ W (1, 1  , 3)$, $ W (0, 1  , 3)$, $ W (1, 1  , 2)$,   $W(0, 1  , 2) $ and
\begin{multline} \label{eq.andrz} 
 W(1, 1  , 3)  + {\rm{span}} \{ z \}   = \langle f u_2 , f^3 u_3 , u_1 + f^2 u_3  \rangle
=  \langle f u_2 , u_1 + f^2u_3  \rangle = 
\\ \langle f u_2 , u_1 + f^2u_3  \rangle ^c . 
\end{multline} 
In the case $ ( \vec \mu,  \vec r  ) = \big(  (0, 2  , 2) ,   (1, 2  , 3) \big) $ 
the interval $ [  W ( \vec r  ) , W ( \vec \mu  ) ] $ consists of 
$ W  (1, 2  , 3) $,  $ W  (0, 2  , 3)  $, $ W  (1, 2  , 2) $, $ W (0, 2  , 2) $ and
\begin{multline*}
W  (1, 2  , 3) +  {\rm{span}} \{ z \}   = \langle f^2 u_2 , f^3 u_3 , u_1 + f^2 u_3  \rangle
=  \langle f^2 u_2 , u_1 + f^2 u_3  \rangle = 
\\
\langle  u_1 + f^2 u_3  \rangle  ^c = \langle  z \rangle  ^c .
\end{multline*} 
}}
\end{example} 

To refine Theorem~\ref{thm.main} we make use of matrices in 
column reduced echelon form. 
Let  \,$ D_{\vec \mu _J}  $ be the  $k$-dimensional vector space in  \eqref{eq.dmjar},
let $ \mathcal Z  $ denote  the lattice of subspaces of  
$ D_{\vec \mu _J} $ 
and let  $ \mathcal M _c $  be the set of $k \times k$ matrices 
in  column reduced echelon form. 
Recall   that a matrix  
 is in column reduced echelon form  if it has the following properties.
(i) The first non-zero  entry  in each column (as we go down)  is a $1$. (ii) 
 These ``leading $1$s'' occur  further down  as we go to the right of  the matrix;
(iii)  In the row of a leading $1$ all other entries are zero. 
To  a matrix   $ M \in K ^{k \times k} $ we associate the subspace 
\begin{multline} \label{eq.bnea} 
Z(M) =  {\rm{span}}  \{z_1 , \dots , z_k \}  \quad \text{with}  \quad 
\\    ( z_1 ,  z_2 , 
\dots  , z_k )  =   \big(
 f^{\mu_{i_1}}   u_{i _1} ,   f^{\mu_{i_2}}   u_{i _2}  , 
\dots ,   f^{\mu_{i_k}}   u_{i _k}   \big)   M .
\end{multline} 
Thus  $ Z \in  \mathcal  Z  $ if and only if $ Z = Z(M) $ for some 
$ M \in K ^{k \times k} $. If $ M_c $ is the   column reduced echelon form of $M$
then $ Z(M) =  Z(M_c) $. Uniqueness of  $ M_c$  
implies that  the mapping 
$ M_c \mapsto Z(M_c) $ is a bijection  from  $ \mathcal M_c  $ onto~$ \mathcal Z $.

The assumptions in  the following theorem are those of 
Theorem~\ref{thm.main}.

\begin{theorem} \label{thm.mzl}
Let $ M \in  K ^{k \times k} $ be in column reduced echelon form
and let $Z(M) $ be the associated 
subspace such that 
$X(M)  = W( \vec r) \oplus Z(M)  $  is a characteristic subspace in
$[ W ( \vec r) , W(\vec \mu ) ]$. 
\begin{itemize}
\item[ {\rm{ (i) }} ]   
$X(M) $ is hyperinvariant if and only if each nonzero
column of $M$ contains exactly one entry $1$. 
\item[ {\rm{ (ii) }} ]  
We have \,$  X(M) _H   = W( \vec r) $\,  
if and only if each nonzero column of $M$ has at least two entries equal to $1$.
\item[ {\rm{ (iii) }} ]  
We have 
\,$X(M)^h = W( \vec \mu) $\, 
if and only if  each row of  $M$ has at least one entry equal to  $1$. 
\end{itemize}  
\end{theorem} 

\begin{proof}
Let $ i_s \in J$. 
Because of  $  Z(M)  \subseteq D_{\vec \mu _J } $ 
we have either   $ Z(M)  \cap  \langle  u_{i_s}  \rangle  = 0 $
or 
\beq \label{eq.ktwq} Z(M)  \cap  \langle  u_{i_s}  \rangle  = 
 {\rm{span}}\{ f^{\mu_{i_s}  }  u_{i_s}  \} , 
\eeq 
and similarly either   \,$ \pi _{i_s}  Z(M)  = 0 $\,  or 
 \,$ \pi _{i_s}  Z(M) =  {\rm{span}}\{ f^{\mu_{i_s}  }  u_{i_s}  \} $.
We note that \eqref{eq.ktwq} holds if and only if  the $s$-th column of 
the matrix $M$  contains exactly one entry $1$ (in row $s$). 
Moreover,  \,$ \pi _{i_s}  Z(M)  = 0 $\, holds if and only if the $s$-th row
is the zero row. 
Suppose \,$\dim Z(M) = \rank M = q $.  Then 
$ M \in \mathcal M_c $ implies 
$ M =  \bpm \tilde M &  0 _{k \times (k -q )} \epm  $ and    $\rank \tilde M = q  $.

(i) Each nonzero column of $ M$ contains exactly one entry~$1$
if and only if 
\,$ \Pi  ^{-1}      M  =    \diag ( I_q ,  \, 0 ) $\,  
for some  permutation matrix  $ \Pi$. This is  equivalent to 
\begin{multline*}  
 (  f^{\mu_{i_1}  }  u_{i_1} ,
\dots ,   f^{\mu_{i_k}  }  u_{i_k}  )  M = 
 (  f^{\mu_{i_1}  }  u_{i_1} ,
\dots ,   f^{\mu_{i_k}  }  u_{i_k}  )  \Pi    \diag ( I_p ,  \, 0 ) 
 = \\
 (   f^{  \mu_{\tau_1}  }  u_{\tau_1}  ,    \dots , f^{  \mu_{\tau_q}  }  u_{\tau_q}  , 
 0, \dots , 0 )
 \end{multline*} 
with $T =  \{\tau_1, \dots, \tau_q \} \subseteq  J =  \{i_1, \dots , i_k\}$. 
Now we apply   Theorem~\ref{thm.main}(ii).

(ii)  In the following let $X = X(M) $.
 From 
\[
X _H = \sum\nolimits _{j= 1}^m  ( X  \cap  \langle  u_j \rangle  ) = 
W(\vec r) +  \sum \nolimits _{s = 1} ^k  
\big( Z(M)   \cap  \langle  u_{i_s}  \rangle \big) 
\]
follows that $X _H  = W(\vec r)  $ is equivalent to 
\beq \label{eq.nzmrs}
Z(M)    \cap  \langle  u_{i_s}  \rangle = 0 , \, \, s = 1, \dots k . 
\eeq 
Condition \eqref{eq.nzmrs} holds if and only if 
$M$ does not contain a nonzero column with exactly one entry $1$. 

(iii) 
From  \,$X^h = \sum \nolimits _{j=1} ^m \pi _j X = W( \vec r ) + 
 \sum \nolimits _{s=1} ^k \pi _{i_s}  Z(M) $\, 
follows that 
$X^h   = W(\vec \mu )  $ is equivalent to 
\[ 
  \pi _{i_s}   Z(M)  = 
    \langle   f^{\mu _{i_s}}  u_{i_s}  \rangle  , \, \, s = 1, \dots k ,
\]
that is, $M$ has no zero row.  
\end{proof}




\begin{example}
{\rm{We
refer to Example~\ref{ex.spcd} and  consider the case $ J = \{1, 2, 3 \} $ with
\,$ \vec \mu = (0,1,2) $. In that case  we have \,$  \vec r = (1, 2, 3)$
and  therefore 
$D_{\vec \mu} =    {\rm{span}} \{ u_1,  f u_2, f^2 u_3\} $. 
We apply Theorem~\ref{thm.mzl}  to determine the subspaces 
$X$ with 
\beq \label{eq.hou}
X_H =    W(  \vec r  ) \quad \text{and} \quad X^h = W( \vec \mu )  . 
\eeq 
The 
two 
matrices  $M_1$  and $M_2$
 that simultaneously satisfy the conditions in \mbox{Theorem~\ref{thm.mzl}(ii)-(iii)}
are 
\[
 M_1 = \bpm  1  & 0 & 0 
\\ 1  & 0 & 0 
\\ 1  & 0 & 0  \epm 
\;\, \text{with} \;\:  Z(M_1) = {\rm{span}}\{ u_1 + f u_2 + f^2u_3 \} 
\]
and 
\[
 M_2 = 
\bpm  1 & 0 & 0 \\
\ 0 & 1 & 0   \\ 1 & 1 & 0   \epm   
\;\, \text{with}  \:\, \,   
Z(M_2) = {\rm{span}} \{ u_1 + f^2u_3,   f u_2  + f^2u_3  \} .
\]
The corresponding characteristic subspaces
$ X_i =   W(  \vec r  ) + Z(M_i)  $, $i = 1, 2 $, 
 are 
\[
 X_1 = G  =  \langle u_1 +f u_2 + f^2 u_3  \rangle^c 
\] 
 and 
\,$ X_2 = F =  \langle u_1 + fu_2 , fu_2 + f^2 u_3  \rangle ^c  $.
Hence, 
$X = G $ and $ X = F $ 
 are the only  elements of  $  [  W( \vec r)  ,W(  \vec \mu ) ]$
that satisfy  \eqref{eq.hou}. 
}}
\end{example} 

\section{Extensions}  \label{sbs.jso} 

In Theorem~\ref{thm.wzmn} 
we have seen that   for a given set $J$ 
 a pair  of $m$-tuples  $ ( \vec r , \vec \mu )  $ 
satisfies 
\beq \label{eq.inrct} 
 \vec \mu  \in  \mathcal{L}( \vec t\, )     \quad \text{and}  
  \quad \vec r   \:   =    \:  \vec \mu     + \sum \nolimits  _{i \in J} \vec e_i   
\in  \mathcal{L}( \vec t\, )    
\eeq  
only if the inequalities  \eqref{eq.krnap}  hold.  
In this section we show that 
 \eqref{eq.krnap} is sufficient for   the existence of such a pair.  We use this
fact for  the construction of characteristic non-hyperinvariant subspaces. 
Let 
\beq \label{eq.jxtym} 
J= \{i_1, \dots, i_k \}     \subseteq  \{1, \dots , m \} , 
\;   i_1 < \cdots < i_k,  \: 2 \le k .
\eeq 
Set 
$ \vec t_J = ( t_{i_1} , \dots  , t_{i_k} )  $.  
Suppose $ \vec  \mu _J  = ( \tilde \mu _{i_1}, \dots , \tilde \mu _{i_k} )
\in    \: \mathcal{L}( \vec t_J \, )     $
and   $ \vec \mu = (  \mu _1, \dots ,  \mu _m ) \in    \: \mathcal{L}( \vec t\, )     $.
We call $ \vec \mu $  an  {\em{extension}} of  $ \vec  \mu _J$
if 
\beq \label{eq.gowa} 
\mu_{i_s} = \tilde \mu _{i_s} ,  \:  s = 1 , \dots , k . 
\eeq 
Let \,$\mathcal E( \vec  \mu _J )  $ be the set   of all extensions 
  of $\vec \mu _J $.   It follows from Lemma~\ref{la.njfos} 
below that  \,$\mathcal E( \vec  \mu _J )  $  is nonempty. 
Since   \,$\mathcal E( \vec  \mu _J )  $ is a sublattice 
of~$ \mathcal{L}( \vec t\, )     $ 
there exists a maximum 
element of  \,$  \mathcal E( \vec  \mu _J )  $,  
which we 
call  the {\em{maximum extension}}. 

Suppose $ \vec  \mu  \in  \mathcal E( \vec  \mu _J )  $.
We take a closer look at the entries of  $ \vec  \mu  $. 
If \, 
$  i_s  \le  j \le  i_{s+1}  $\, then $  \tilde \mu_{i_s } \le  \mu_j \le   \tilde \mu_{ i_{(s+1)}  }  $
and 
\beq \label{eq.prms} 
 t_{i_s }  - \tilde \mu_{i_s }  \le   t_j -  \mu_j \le    t_{ i_{(s+1)}  }  
  - \tilde \mu_{i_{(s+1) } } .   \eeq
Since \eqref{eq.prms}   is equivalent to    
\[
     t_j - (  t_{ i_{(s+1)} } -  \tilde \mu_{i_{(s+1) } } )  \le \mu _j \le 
t_j- (t_{i_s} - \tilde \mu_{i_s}) 
\]
we obtain 
\beq \label{eq.bdugl} 
 \mu _j \le \min \{   t_j - 
(  t_{ i_s }  - \tilde \mu_{i_s }    ) ,  \tilde \mu_{ i_{(s+1)}  } \}  .  
\eeq 
If $  1  \le  j \le  i_1  $ then  \,$0 \le   \mu_j \le   \tilde \mu_{ i_1 }  $
and  \,$ 0 \le t_j - \mu_j  $.   
Hence 
\beq \label{eq.lksu}  
   \mu_j  \le \min \{ t_j ,  \tilde \mu_{ i_1 }  \}. 
\eeq 
If $ i_k \le j  $ then   $  \tilde \mu _{i_k}  \le  \mu _j $  and 
$ t_{i_k} -   \tilde \mu _{i_k}  \le t_j -  \mu _j  $,
and therefore 
\beq \label{eq.grtsa} 
   \tilde \mu _{i_k}  \le  \mu _j  \le t_j - ( t_{i_k} -   \tilde \mu _{i_k} ) .
\eeq 

\medskip

\begin{lemma} \label{la.njfos} 
Assume  \eqref{eq.jxtym}. 
Suppose  
$ \vec \mu_{J}   = ( \tilde \mu _{i_1} , \dots , \tilde  \mu_{i_k} ) 
\in  \mathcal{L}( \vec t \, ) $, that is, 
\beq  \label{eq.ndwo8}
0 \le \tilde  \mu_{i_1} \le \dots \le \tilde \mu_{i_k}  \quad
 \mbox{and} 
\quad  0 \le   t_{i_1} - \tilde \mu_{i_1} \le  \dots \le  t_{i_k}- \tilde \mu_{i_k} .
\eeq
Define 
\beq \label{eq.mjt8} 
\mu_j = 
\begin{cases} 
\min \{ t_j ,  \tilde \mu_{ i_1 }  \} 
 & \hbox{if} \quad 1 \le  j  \le i_1
\\
\min \{  t_j- (t_{i_s} - \tilde \mu_{i_s})  ,  \, \tilde \mu_{i_{(s+1)}}  \} &
 \hbox{if} \quad   i_s  \le j  \le i_{s+1} , \, 
 s =1, \dots, k-1
\\
t_j - ( t_{i_k} -  \tilde  \mu_{i_k} ) & \hbox{if}  \quad   i_k  \le j \le m . 
\end{cases}  
\eeq 
\begin{itemize}
\item[ {\rm{ ($\alpha$) }} ]  
Then  $\vec \mu $ is the maximum extension of  $ \vec \mu_{J}$. 
\item[ {\rm{ ($\beta$) }} ]   
If $J \subset I_u$ and 
\beq \label{eq.ngftoo2}
0 \le  \tilde  \mu _{i_1} < \dots < \tilde \mu _{i_k} 
  \quad \textrm{and} \quad
0 < t_{i_1}  - \tilde  \mu _{i_1}  < \dots <  t_{i_k} -  \tilde \mu _{i_k} ,
\eeq 
then \,$ \vec r   =   \vec \mu     + \sum \nolimits  _{i \in J} \vec e_i  
\in \mathcal{L}( \vec t \, )$. 
\end{itemize} 
\end{lemma}

\begin{proof} 
($\alpha$) 
To prove  that $\vec \mu  $ is an extension of $ \vec \mu_{J} $
we  have to show that the conditions  \eqref{eq.gowa} 
and 
\beq \label{eq.stumg8} 
0  \, \le \,  \mu_j  \, \le \,  \mu _{j+1} 
\eeq 
and 
\beq  \label{eq.zwtumg8} 
0  \, \le \,  t_j - \mu_j  \, \le \,  t_{j +1} - \mu _{j+1}  , 
\eeq 
$  j = 1, \dots , m - 1 $, 
are satisfied.  
We consider different cases. 
\begin{itemize} 
 \item[ {\rm{ (i) }} ]   Case $j = i_s$, $ s \in \{1, \dots , k\}$.
Then \eqref{eq.mjt8}  yields $ \mu _{i_s} = \tilde \mu _{i_s} $.
 \item[ {\rm{ (ii) }} ] Case   $ 1 \le j < i_1  $.  Then $ \mu _j  \ge 0 $
and $ t_j - \mu_j \ge 0 $.  From 
\,$ t_j \le t_{j+1} $\,  follows 
\[  
\mu_j \, = \, \min\{ t_j,  \tilde \mu_{i_1} \} \,  \le  \, 
\min\{ t_{j+1},  \tilde \mu_{i_1} \} \, = \, \mu_{j+1} .  
\] 
\begin{itemize} 
\item[ {\rm{ (I) }} ]   Case $\mu _j =  t_j $. Then 
\, $ t_j - \mu _j = 0 \le  t_{j+1} - \mu_{j+1} $ 
such that \eqref{eq.zwtumg8}  is satisfied. 
\item[ {\rm{ (II) }} ] Case $\mu _j =  \tilde \mu_{i_1} $.
Then \,$t_{j+1} \ge t_j \ge  \tilde   \mu_{i_1}$,   which implies 
\[
\mu_{j+1}  = \min\{ t_{j+1},  \tilde \mu_{i_1} \} = \tilde  \mu_{i_1} =\mu _j  
\]
and  \,$ t_{j+1} - \mu_{j+1} \,  \ge  \,  t_j  -  \mu _j  $. 
\end{itemize} 
\item[ {\rm{ (iii) }} ] Case $ i_k \le j \le m $. 
Then \,$ t_{i_k}   \le t_j $ and 
$ t_j   -  \mu _j  = t_{i_k} -  \mu_{i_k}  $ imply \eqref{eq.stumg8} 
 and \eqref{eq.zwtumg8}, 
respectively. 
\item[ {\rm{ (iv) }} ] Case  $ i_s  \le j < i_{(s+1)} $, $s \in \{ 1, \dots , k -1 \}$. 
\begin{itemize} 
\item[ {\rm{ (I) }} ]  Case $t_j \le  t_{j+1} \le  t_{i_s } + (   \tilde \mu_{i_{(s+1)}} -  
 \tilde \mu_{i_s}) $. 
Then  
\,$ \mu _j  = t_j-  (t_{i_s} - \tilde \mu_{i_s})  $\, and  
\,$ \mu_{j+1} =  t_{j+1}  -  ( t_{i_s} - \tilde \mu_{i_s}) $.
Hence 
\,$   \mu _j  \le \mu_{j+1} $.  
Moreover,  $ t_j   -  \mu _j   =  t_{j+1}  -   \mu_{j+1}  =
   t_{ i_s} -  \tilde \mu_{i_s}  $.

\item[ {\rm{ (II) }} ]  
Case  $   t_{ i_{(s+1)}} > 
t_{j+1} \ge   t_j \ge   t_{i_s } + ( \tilde  \mu_{ i_{(s+1)}} -   \tilde  \mu_{i_s} ) $.
Then 
$ \mu _j  =   \mu _ {j  + 1 } =   \tilde \mu_{i_{(s+1)}}  $,  which implies \eqref{eq.zwtumg8}, 
\item[ {\rm{ (III) }} ]      Case  
$ t_j \le     t_{i_s } + (   \tilde \mu_{i_{(s+1)}} -   \tilde \mu_{i_s})  \le t_{j+1}  $.
Then 
$ \mu _j =   t_j  -  ( t_{i_s }  -   \tilde \mu_{i_s}) \le  \tilde \mu_{i_{(s+1)}} =  \mu _ {j  + 1 }  $.   
Hence we obtain 
\[  t_j  -  \mu _j =  t_{i_s }  -   \tilde \mu_{i_s} \le t_{j+1}   - \tilde \mu_{i_{(s+1)}} 
= t_{j+1}   -  \mu _ {j  + 1 } . 
\]
\end{itemize} 
\end{itemize} 
From \eqref{eq.bdugl}  - \eqref{eq.grtsa}   
we conclude that $ \vec \mu $ is the maximum element 
of  $ \mathcal E( \vec  \mu _J )  $. 

($\beta$)  
 If  \,$ i_s \in I_u $\,  then the corresponding elementary divisor
$ \lambda ^{t_{i_s}} $ is unrepeated, 
and therefore
\beq \label{eq.urpa} 
 t_{(i_s  -1 )} <  t_{i_s} < t_{(i_s +1 )} ,  \;  i = 1 , \dots , k. 
\eeq  
We have $ \vec \mu  \in  \mathcal{L}( \vec t \, ) $ and 
\[
 r_j  =
\begin{cases} 
                        \tilde    \mu _j + 1 \:\: \:  \text{if} \:\: \:  j \in J \\
 \mu _j   \:\:  \:\text{if}  \:\: \: j \notin J . 
\end{cases}
\] 
Hence in order to prove $ \vec r  \in  \mathcal{L}( \vec t \, ) $
we have to show that 
\beq \label{eq.mlkn} 
\tilde  \mu_{i_s} < \mu_{(i_s + 1) } 
\eeq
 and  
\beq  \label{eq.jtdb} 
   t_{(i_s -1) } - \mu_{(i_s -1) }  < t_{i_s } -  \tilde \mu_{i_s } , 
\eeq 
$ s = 1,  2, \dots, k$. 
In the case $ s = k $  definition \eqref{eq.mjt8} implies
\[
\mu _{i_k + 1} = (  t_{i_k + 1} - t_{i_k}  ) + \tilde \mu _{i_k} .
\]
Then \eqref{eq.urpa} yields  $ \mu _{i_k + 1}  > \tilde \mu _{i_k}  $.
In the case   $ s < k $ we have 
\[  \mu_{(i_s +1)} =
  \min \{ t_{(i_s +1 )} - (t_{i_s} - \tilde \mu_{i_s}), \tilde \mu_{i_{(s+1)}}  \}.
\]
If  $ \mu_{(i_s +1)} = t_{(i_s +1 )} - (t_{i_s} - \tilde \mu_{i_s})  $ then
\eqref{eq.urpa}  yields \eqref{eq.mlkn}.  If $ \mu_{(i_s +1)} =  \tilde \mu_{i_{(s+1)}} $
then \eqref{eq.mlkn} follows from  the strict inequality  \,$ \tilde \mu_{i_s}
< \tilde \mu_{i_{(s+1)}}$.

It remains to deal with \eqref{eq.jtdb}.  Let 
$ s > 1 $. Then \,
 $    i _{s -1}  \le  \, i_s -1 \,  <  i_s  $\,  
implies 
\[
 \mu _{(i_s -1)} =  \min\{    t_{ (i_s -1) } - t_{ i_{(s - 1)}  }  + 
  \tilde \mu _{i_{(s-1)} }   , \,  \tilde \mu_{i_s}   \}  .
\] 
Suppose 
$ \mu _{(i_s -1)}  =   t_{ (i_s -1) } - t_{ i_{(s - 1)}  }  + 
  \tilde \mu _{i_{(s-1)} } $. 
Then   
\,$
   t_{ i_{(s - 1)}  }  -   \tilde \mu _{i_{(s-1)} }   < t_{i_s } -  \tilde \mu_{i_s } $\,
implies 
\begin{multline*}
 t_{(i_s -1) } - \mu_{(i_s -1) }   = t_{(i_s -1) } -
[
  t_{ (i_s -1) } - t_{ i_{(s - 1)}  }  + 
  \tilde \mu _{i_{(s-1)} } 
]  = 
\\ t_{ i_{(s - 1)}  }  -
  \tilde \mu _{i_{(s-1)} }   < t_{i_s } -  \tilde \mu_{i_s } .  
\end{multline*} 
Suppose $ \mu _{(i_s -1)}  =    \tilde \mu_{i_s}  $. 
Then \eqref{eq.urpa}  implies  
\[   t_{(i_s -1) } - \mu_{(i_s -1) }   = t_{(i_s -1) } -  \tilde \mu_{i_s}  
<    t_{i_s  } -  \tilde \mu_{i_s} .  
\]  
Let $ s = 1 $.  Then
  \,$ \mu _{( i_1  -1 )}  = \min\{t_{(i_1 -1)}, \tilde \mu _{i_1}\}$. 
If $\mu _{( i_1  -1 )} = \tilde \mu_{i_1}$,
then   
\beq \label{eq.nfsi} 
t_{(i_1 -1) } - \mu _{(i_1 -1)} < t_{i_1} - \tilde \mu _{i_1} 
\eeq 
 follows   from  \eqref{eq.urpa}.   If  \,$\mu _{( i_1  -1 )} =  t_{(i_1-1)}$
then  the strict inequality  $ 0 < t_{i_1} - \tilde{\mu}_{i_1}$ 
in \eqref{eq.ngftoo2} implies \eqref{eq.nfsi}.   
\end{proof}

We note without proof that the minimum
element  
 of   $\mathcal{E}(\vec \mu_J) $
is given by 
\beq \label{eq.mnxtn} 
\mu_j = 
\begin{cases}  \max\{0, \,  t_j - ( t_{i_1} -  \tilde  \mu_{i_1} ) \} 
& \hbox{if} \quad 1 \le  j  \le i_1
\\
\max  \{  t_j- (t_{i_{(s+1)}} - \tilde \mu_{i_{(s+1)}})   , \,
 \tilde \mu_{i_s}  \} &
 \hbox{if} \quad   i_s  \le j  \le i_{s+1} , \, 
 s =1, \dots, k-1
\\
\tilde \mu _{i_k}   & \hbox{if}  \quad   i_k  \le j \le m . 
\end{cases}  
\eeq

\medskip 

The next theorem provides  an existence result.
It shows that to a given admissible set $J$ there exists 
a characteristic non-hyperinvariant subspace $X$ 
such that \mbox{$J(X) = J$.} 

\medskip

\begin{theorem}  \label{thm.mxstr}
Assume 
\beq \label{ex.mtud}  
J= \{i_1, \dots, i_k \}    \subseteq I_u ,  \: i_1 < \cdots < i_k, \;
2 \le k .
\eeq
Suppose 
$\vec \mu_{J}   = ( \tilde  \mu _{i_1} , \dots ,  \tilde   \mu_{i_k} )     \in      
\mathcal{L}( \vec t _J  \, )  $ and let 
$\vec \mu   = (\mu_1, \dots , \mu_m) 
 \in \mathcal{L}( \vec t \, )  $ be  the maximum   extension of
$\vec \mu_{J} $.  
Set 
$\vec r   \:   =    \:  \vec \mu     + \sum \nolimits  _{i \in J} \vec e_i    $. 
Then 
the following statements are equivalent. 
\begin{itemize} 
\item[\rm{(i)}] 
The entries of $ \vec \mu _J$ satisfy  the inequalities 
$ 0 \le \tilde  \mu_{i_1} < \dots < \tilde \mu_{i_k}$\, 
and the strict  inequalities 
\,$  0 <   t_{i_1} - \tilde \mu_{i_1} < \dots < t_{i_k}- \tilde \mu_{i_k} $.
\item[\rm{(ii)}] There exists a characteristic non-hyperinvariant
subspace  $X$ with hyperinvariant frame 
 $(X_H ,  X^h ) = ( W(\vec r) , W (\vec \mu ))  $. 
\end{itemize}  
\end{theorem} 

\begin{proof}  
The implication (i) $\Ra$ (ii) follows form Theorem~\ref{thm.main}
and Lemma~\ref{la.njfos}(ii)
and the implication 
 (ii) $\Ra$ (i)    is  consequence of  Theorem~\ref{thm.wzmn}.  
 \end{proof}

\medskip 

In the case of the  maximum extension $\vec \mu$ 
of $\vec \mu_{J} $ 
one can give a concise description of the  subspaces in
  $   [ W( \vec r) ,  W(\vec \mu) ] $
 in Theorem~\ref{thm.main}.  

\medskip

\begin{theorem} \label{thm.nevr}  
Assume \eqref{ex.mtud}  
and 
\beq \label{eq.ngfto}
0 \le  \mu _{i_1} < \dots < \mu _{i_k} 
  \quad \textrm{and} \quad
0 < t_{i_1}  -  \mu _{i_1}  < \dots <  t_{i_k} -  \mu _{i_k}.
\eeq
Let $\vec \mu   = (\mu_1, \dots , \mu_m)  $
be the maximum  extension of
$\vec \mu_{J}   = (  \mu _{i_1} , \dots ,   \mu_{i_k} )  $
and let 
$ \vec r   =   \vec \mu     + \sum \nolimits  _{i \in J} \vec e_i $.
Let 
\,$
Z =  {\rm{span}}\{z_1, \dots , z_q \} $\, 
 be   a $q$-dimensional subspace of 
\beq \label{eq.dschwd} 
D_{\vec \mu _J}  = {\rm{span}} \{ f^{ \mu _{i_1} } u  _{i_1},  
\dots , f^{ \mu _{i_k} }  u_{i_k} \} . 
\eeq
If \,$ X = W( \vec r) \oplus Z $ 
satisfies  \,$X^h =  W( \vec \mu) $
then   
\,$ X =  
\langle z_1, \dots , z_ q\rangle  ^c$. 
\end{theorem}


\begin{proof} 
It is obvious that  $   Z  ^c   \subseteq X  $.
Because of $ Z  \subseteq X $  the converse inclusion  $ X   \subseteq  Z ^c  $
is equivalent to  
\[
 W(\vec r) = \langle f^{r_1} u_1 , \dots , f^{r_m} u_m 
\rangle   \subseteq   Z ^c .
\]
Set 
\,$ \mathfrak Z = \{ z_1, \dots , z_q   \}$.    
To  check  that $  f^{r_j} u_j \in  Z ^c $ we separately deal with 
 different cases of $j $.   
In  each of the cases (i) - (iv) below 
we choose   suitable  automorphisms  $\alpha   \in   {\rm{Aut}}(f,V) $ 
such that 
$\alpha  z     =  z   + f^{r_j } u_j $ for some $ z  \in   \mathfrak Z  $. 
Then $  f^{r_j } u_j  \in   \langle   z \rangle  ^c   \subseteq  Z ^c $.
The assumption  $ X ^h = W(\vec \mu ) $ implies that for each $j \in J$ 
 there exists an element  $ z \in   \mathfrak  Z  $  
such that  $\pi _ j  z \ne 0 $.
Let $\alpha( u_j , u^{\prime} _j) $ denote the automorphism that
exchanges the generator $u_j$ by $  u^{\prime} _j$.
Recall that   $r_j =  1 + \mu _j $ if $j  \in J$ and $r_j = \mu _j $ if $j \notin J$. 
\begin{itemize} 
 \item[ {\rm{ (i) }} ]  Case $j = i_s$, $ s \in \{1, \dots , k\}$.
If $ z \in \mathfrak Z  $ satisfies $\pi _ {i_s}   z \ne 0 $ 
then $ \alpha =  \alpha (  u_{i_s} , u_{i_s}  + f u_{i_s} ) $ 
yields  $\alpha  z     =  z +   f^{1 + \mu_{i_s} }  u_{i_s}$. 
 \item[ {\rm{ (ii) }} ] Case   $ 1 \le j < i_1  $.  
If $\mu_j  =  r_j   = t_j $ then it is obvious that 
$f ^{r _j}  u_j = 0 \in Z^c$.  If  $\mu_j =  r_j = \mu_{i_1} $  
 then $i_1 \in I_u $ implies 
$ \e(u_j) < \e(u_{i_1} ) $ and $ \e(u_{i_1} ) = \e (   u_ {i_1}   + u_j )$. 
Hence, if  $ z \in \mathfrak Z  $ satisfies $\pi _ {i_1}   z \ne 0 $ then
 \,$ \alpha  = \alpha ( u_{i_1} ,  u_ {i_1}   + u_j )$\,   yields 
\,$ \alpha z =  z +   f^{ \mu_{i_1}}  u_j  = z +   f^{ r_j }  u_j $.  
\item[ {\rm{ (iii) }} ] Case $m \ge  j > i_k $.
Then \,$ t_j > t_{i_k} $ and  
$ \mu _j =  t_j  -  t_{i_k} +  \mu_{i_k} $. 
We have 
$ \e(u_{i_k}) =  \e(u_{i_k} +  f^{t_j - t_{i_k} } u_j)  $. 
If $ z \in \mathfrak Z    $ satisfies $\pi _ {i_k}   z \ne 0 $ then
 \,$ \alpha  = \alpha (   u_{i_k} ,  u_{i_k}  +  f^ { t_j - t_{i_k} } u_j ) $
yields 
$ \alpha      z  = z + f^{\mu _{i_k} + ( t_j - t_{i_k}) } u_j $.  
\item[ {\rm{ (iv) }} ] Case  $ i_s < j < i_{(s+1)} $, $s \in \{ 1, \dots , k -1 \}$. 
\begin{itemize} 
\item[ {\rm{ (I) }} ]  Case $t_j \le   t_{i_s } + (   \mu_{i_{(s+1)}} -  \mu_{i_s}) $.
Then  $ \mu _j = t_j - (   t_{i_s }  -   \mu_{i_s} )  $.
If $ z \in \mathfrak Z    $ satisfies  $\pi _ {i_s}   z \ne 0 $ then
 \,$ \alpha  = \alpha (  u_{i_s} ,  u_{i_s} + f^{  t_j - t_{i_s}  }  u_j )$\, 
yields  \,$  \alpha z = z + f^{ \mu _{i_s}  +  t_j - t_{i_s}   } u_j  $. 
\item[ {\rm{ (II) }} ]  Case
 $t_j \ge   t_{i_s } + (  \mu_{i_{(s+1)}} - \mu_{i_s} ) $.  Then 
$ \mu _j = \ \mu_{i_{(s+1)}}  $.
If $ z \in \mathfrak Z    $ satisfies  $\pi _ {i_{(s+1)}}   z \ne 0 $ then
  \,$ \alpha  = \alpha (  u_{i_{(s + 1)}} , 
 u_{i_{(s + 1)}} + u_j ) $\,   yields \,$  \alpha z = z +  f^{\mu _{i_{(s+1)} } }u_j  $. 
\end{itemize} 
\end{itemize} 
\end{proof}

\begin{corollary} \label{cor.nrzln} 
Assume    \eqref{ex.mtud} and  \eqref{eq.ngfto}.
Let 
\beq \label{eq.zcrg}
z =  f^{\mu_{i_1} } u  _{i_1}  + \cdots +  f^{\mu_{i_k} } u  _{i_k} .
\eeq 
If $ \vec \mu $ is the maximum extension of  
$\vec \mu_J=   (  \mu _{i_1} , \dots ,   \mu_{i_k} )  $
and $ \vec r = \vec \mu + \sum _{s = 1} ^k \vec e_{i_s} $
then   
\,$
X(z)  =  W(\vec r) + {\rm{span}} \{ z \} $\, 
is a characteristic  non-hyperinvariant  subspace and 
  $X(z)  =  \langle z \rangle ^c $. 
\end{corollary}

\begin{proof} 
We have $Z=  {\rm{span}} \{ z \} = Z(M)    $  with 
\[
  M = \bpm e  &   0 _{k \times (k - 1)} \epm  \quad \text{and} \quad e = (1, 1, \dots , 1) ^T . 
\] 
Theorem~\ref{thm.mzl} implies $ X^h = 
 W(\vec \mu) $. Hence we can apply Theorem~\ref{thm.nevr}.  
Moreover, 
$ X_H= 
 W(\vec r) $.  Therefore $X$ is not hyperinvariant. 
\end{proof}

\begin{example} 
{\rm{ 
Let $(V,f) $ be given by  \eqref{eq.vedr}.  
Referring to Example~\ref{ex.spcd} we consider the case 
 $ J = \{1, 3 \} \subseteq I_u = \{1, 2, 3\} $ with 
$\vec \mu _J =(\mu_1, \mu_3) =  (0,2) $.
Two extensions $\vec \mu $ of $\vec \mu _J $ 
satisfy  \,$ \vec  r =  \vec \mu + (1 , 0 , 1)  \in \mathcal{L}( \vec t \, ) $,
namely $ \vec \mu =  \vec \mu _1 =  (0, 1  , 2) $ and $ \vec \mu =
  \vec \mu _ 2 =  (0, 2  , 2)  $. 
The corresponding triples $ \vec r $ are 
$ \vec r_1 =  (1 , 1 , 3 )$  and $ \vec r_2 = (1, 2  , 3)  $.
We have $  \vec \mu _1 \preceq   \vec \mu _2 $.
Thus $ \vec \mu _2  $ is the maximum extension in $ \mathcal E_J $. 
Define $z = u_1 + f^2  u_3$ 
according to  \eqref{eq.zcrg}.  
Then 
\[
X(z) = 
W (\vec r_2 ) + {\rm{span}} \{ z \} =
 \langle f^2 u_2 , f^3 u_3 , u_1 + f^2 u_3  \rangle =  \langle  z \rangle  ^c 
\]
is a characteristic non-hyperinvariant subspace of $V$.
In the case of $ \vec \mu = \vec \mu _1 $ we recall \eqref{eq.andrz} 
and note that 
\,$ X = W (\vec r_1) + {\rm{span}} \{ z \} $ 
is not the characteristic  hull of a single vector. 
}}
\end{example} 

In \cite{AW7} we studied 
invariant subspaces  that  are the characteristic hull of  
a single vector. 
Using a decomposition lemma due to Baer~\cite{Bae2}
we proved the following result,  which yields part of 
Corollary~\ref{cor.nrzln}.   

\begin{theorem}   \label{thm.prla} 
For a given nonzero $ z \in V $ there exists a generator tuple
$U =(u_1, \dots , u_m) $ such that 
$z$ can be represented in the form
\beq \label{eq.ukog}
z =   f^{\mu_{\rho_1}} u_{\rho_1}+ \dots + f^{\mu_{\rho_k}}  u_{\
rho_k} , 
\eeq
with 
\beq  \label{eq.eugl} 
0 \le  \mu_{\rho_1} <   \dots < \mu_{\rho_k} \quad    
and \quad 
0 < t_{\rho_1} -  \mu_{\rho_1} < 
\dots <  t_{\rho_k}  - \mu_{\rho_k}.  
\eeq 
The following statements are equivalent.
\begin{itemize}
  \item[\rm{(i)}] 
The subspace 
$ X =  \langle z \rangle ^c $ is not hyperinvariant.
 \item[\rm{(ii)}] 
 At least two of the generators $ u_{\rho_i }$ in  \eqref{eq.ukog}
are unrepeated. 
\end{itemize} 
\end{theorem} 
The assumptions \eqref{ex.mtud} and  \eqref{eq.ngfto}
in Corollary~\ref{cor.nrzln}  
imply that in \eqref{eq.zcrg}
all  generators $u_{i_s}$,  $ s = 1, \dots , k$, are unrepeated
and that \eqref{eq.eugl} holds. 
Hence it follows from Theorem~\ref{thm.prla} 
that the characteristic  subspace 
$X(z) = \langle z \rangle ^c $ is not hyperinvariant.

\medskip

We reexamine Shoda's theorem. 
Using Corollary~\ref{cor.nrzln} or 
Theorem~\ref{thm.prla} we refine the 
implication (ii) $\Ra $ (i) in Theorem~\ref{thm.vnpsa}. 

\begin{corollary} \label{cor.shrv}
Let 
 $ \lambda ^R   $ and  $ \lambda ^S $ 
be unrepeated elementary divisors of $f$ such that 
 \,$ R  + 1  < S $. 
Let
 $ u $ and $v$ be generators of $(V, f) $ with $\e(u) = R $ and $
 \e(v) = S $. 
If   the integers $ s$ and  $ q $ satisfy 
\[
0 \le s < q \quad and \quad  0 < R -s < S -q 
\] 
then 
 \,$  X =      \langle f^{s  } u +   
  f^{q } v \rangle ^c  $\, 
is a characteristic subspace of $V$   that is not hyperinvariant.
\end{corollary}

\section{Concluding remarks} 

From Theorem~\ref{thm.fhl} one can deduce properties of
the lattice of  ${\rm{Hinv}}(V,f) $.  It is known  \cite{FHL} that 
${\rm{Hinv}}(V,f) $ is self-dual in the sense that
there exists a bijective map 
$\Lambda :  {\rm{Hinv}}(V,f) \to  {\rm{Hinv}}(V,f) $
such that 
\[ 
\Lambda (W + Y ) = \Lambda (W) \cap \Lambda (Y)  
\quad \text{and} \quad 
\Lambda (W \cap  Y ) = \Lambda (W) +  \Lambda (Y)  \]
for all $W, Y \in {\rm{Hinv}}(V,f)  $.   
It is not difficult to show (see \cite[p.\ 343]{GLR} that 
$ \vec r \in  \mathcal{L}(\vec t \,) $ 
if and only if  $ \overrightarrow{t -r} \in  \mathcal{L}(\vec t \,)  $. 
Hence, if $ \vec r \in  \mathcal{L}(\vec t \,) $ and 
\[ 
W(\vec r \, ) = f ^{r_1}V \cap V[ f^{t_1 - r_1} ]
 \,  + \cdots + \, 
 f^{r_m}V \cap V[ f^{t_m- r_m} ]  \in \mathcal{L}(\vec t \,) 
\]
then 
$ \Lambda \big( W(\vec r \, ) \big) $ is given by 
\[
\Lambda \big( W(\vec r \, ) \big) = 
 W(\overrightarrow{t -r} ) = 
f ^{t_1 - r_1}V \cap V[ f^{r_1} ]
 \,  + \cdots + \, 
 f^{t_m - r_m}V \cap V[ f^{r_m} ] . 
\]
For the moment it is an open problem whether the lattice 
${\rm{Chinv}}(V,f) $ is self-dual, and it remains to clarify 
the lattice structure of  ${\rm{Chinv}}(V,f) $.
A useful tool for such  an investigation will be the concept of hyperinvariant frame, 
which we introduced in this paper.


 \enlargethispage{\baselineskip}

\begin{thebibliography}{99}


\bibitem{AW1}
P.\ Astuti and H.\ K. Wimmer,
Hyperinvariant, characteristic and marked subspaces,
Oper. Matrices 3 (2009), 261--270.

\bibitem{AW2}
P.\ Astuti and H.\ K. Wimmer,
Characteristic and hyperinvariant subspaces over the field 
${\rm{GF}}(2)$, 
Linear Algebra Appl. 438 (2013), 1551--1563.


\bibitem{AW5}
P.\ Astuti and H.\ K. Wimmer,
Linear transformations with characteristic subspaces that are not  hyperinvariant,
Oper. Matrices 8 (2014), 725--745. 
 



\bibitem{AW7}
P.\ Astuti and H.\ K. Wimmer,
Characteristic invariant subspaces  generated by a single vector, 
Linear Algebra Appl.  470 (2015).  81--103.







\bibitem{Bae}
R. Baer,
Types of elements and characteristic subgroups of Abelian groups,
 Proc. London Math. Soc. 39 (1935),
481--514.


\bibitem{Bae2} 
R. Baer, Primary abelian groups and their automorphisms,  Amer. J. Math.
59 (1937),  99--117. 



 \bibitem{BCh}
 K.\ Benabdallah and  B.\ Charles,
 Orbits of invariant subspaces of algebraic linear operators,
 Linear Algebra  Appl. 225 (1995), 13--22.


 \bibitem{BEis} 
K. M.  Benabdallah, B. J. Eisenstadt, 
J. M.  Irwin, and E. W. Poluianov,
 The structure of large subgroups of primary abelian groups,
Acta Math.  Hung.  21   
(1970), 
421--435.
 


\bibitem{FHL}
P.\,A. Fillmore, D.\,A. Herrero, and W.\,E. Longstaff,
The hyperinvariant  
subspace lattice of a linear transformation,
Linear Algebra Appl.~17 (1977), 125--132.


\bibitem{FuI} L.\ Fuchs, Infinite Abelian Groups, Vol. I.,
Academic Press, New York, 1973.


\bibitem{FuII} L.\ Fuchs, Infinite Abelian Groups, Vol. II.,
Academic Press, New York, 1973.


\bibitem{GLR} I.\ Gohberg, P.\ Lancaster, and L.\ Rodman,
 Invariant Subspaces of Matrices with Applications,
Wiley, New York, 1986.
 

\bibitem{Kap}
I.\ Kaplansky,  Infinite Abelian Groups, University of Michigan
Press, Ann Arbor, 1954.



\bibitem{KR}
B.\ L.\ Kerby and  E.\  Rode, Characteristic subgroups of finite abelian groups,
Comm. Algebra  39 (2011),  1315--1343.
%




\bibitem{Kn}
D. E. Knuth, 
Subspaces, subsets, and partitions,  J. Combin. Theory Ser. A 10 (1971),  
178--180.
  
 

\bibitem{Lo}
W.\,E. Longstaff,
A lattice-theoretic description of the lattice of hyperinvariant
subspaces of a linear transformation,
 Can. J. Math. 28 (1976), 1062--1066.


\bibitem{Lo2}
W.\ E.\ Longstaff, Picturing the lattice of invariant subspaces
of a nil\-potent complex matrix, Linear Algebra Appl.\ 56 (1984), 161--168.


\bibitem{Ba}
D. Mingueza, M. E.  Montoro,  and J. R. Pacha, 
Description of characteristic non-hyperinvariant subspaces over the field 
${\rm{GF}}(2)$, Linear Algebra Appl.~439 (2013), 3734--3745.



\bibitem{NSW}
 A. Nijenhuis, A. E. Solow, and H. S. Wilf,
Bijective methods in the theory of finite vector spaces, 
 J. Combin. Theory Ser. A  37 (1984),  \mbox{80--84}. 


\enlargethispage{20mm}

\bibitem{Sh}
K.\ Shoda, \"Uber die charakteristischen Untergruppen einer
endlichen Abelschen Gruppe,
 Math. Zeit. 31 (1930), 611--624.

\end{thebibliography}
 \end{document}